\documentclass[12pt]{article}
\usepackage{dsfont}
\usepackage{latexsym,amsmath,amscd,amssymb,graphics}
\usepackage{enumerate,framed,comment}
\usepackage{color,url}
\usepackage{verbatim}
\usepackage[margin=2.5cm]{geometry}
\usepackage{amsthm}
\usepackage{epsfig,epstopdf,color,enumerate}
\usepackage{caption}
\usepackage{float}
\usepackage{placeins}
\usepackage{graphicx, subfigure, psfrag} 

\usepackage[colorlinks]{hyperref}
\usepackage{url}  %
\usepackage[percent]{overpic}

\usepackage[T1]{fontenc}
\usepackage[colorlinks]{hyperref}
\usepackage{sidecap}
\allowdisplaybreaks

\newcommand{\ri}{\mbox{$\rm i$}}
\newcommand{\rd}{\mbox{$\rm d$}}

\usepackage[all]{xy}

\DeclareMathOperator{\tr}{tr}

\newcommand{\rem}[1]{}
\makeatletter

\@addtoreset{figure}{section}
\def\thefigure{\thesection.\@arabic\c@figure}
\def\fps@figure{h, t}
\@addtoreset{table}{bsection}
\def\thetable{\thesection.\@arabic\c@table}
\def\fps@table{h, t}
\@addtoreset{equation}{section}

\makeatother

\rem{
\textwidth 6.5 truein
\oddsidemargin 0 truein
\evensidemargin .2 truein
\topmargin -.6 truein
\textheight 9.1 in
}

{

\begin{document}

\newtheorem{theorem}{Theorem}[section]
\newtheorem{definition}[theorem]{Definition}
\newtheorem{lemma}[theorem]{Lemma}
\newtheorem{remark}[theorem]{Remark}
\newtheorem{proposition}[theorem]{Proposition}
\newtheorem{corollary}[theorem]{Corollary}
\newtheorem{example}[theorem]{Example}

\def\below#1#2{\mathrel{\mathop{#1}\limits_{#2}}}



\title{
Inexact trajectory planning and inverse problems\\ in the Hamilton--Pontryagin framework}

\author{
Christopher L. Burnett$^{1}$, Darryl D. Holm$^{1}$, David M. Meier$^{1}$
}
{\addtocounter{footnote}{1} 
\footnotetext{Department of Mathematics, Imperial College, London SW7 2AZ, UK. 
\texttt{c.burnett@imperial.ac.uk, d.holm@ic.ac.uk, d.meier09@imperial.ac.uk}
}%
%
\date{}

\maketitle

\makeatother
\maketitle


%

\begin{abstract} 

We study a trajectory-planning problem whose solution path evolves by means of a Lie group action and passes near a designated set of target positions at particular times. This is a higher-order variational problem in optimal control, motivated by potential applications in computational anatomy and quantum control. Reduction by symmetry in such problems naturally summons methods from Lie group theory and Riemannian geometry. 
A geometrically illuminating form of the Euler--Lagrange equations is obtained from a higher-order Hamilton--Pontryagin variational formulation. In this context, the previously known node equations are recovered with a new interpretation as Legendre--Ostrogradsky momenta possessing certain conservation properties. Three example applications are discussed as well as a numerical integration scheme that follows naturally from the Hamilton--Pontryagin principle and preserves the geometric properties of the continuous-time solution.
\end{abstract}
\newpage
\tableofcontents


\section{Introduction}
\paragraph{The purpose of this paper.} This paper is concerned with the analysis of a class of higher-order trajectory planning problems that are important in a wide range of contexts, ranging from computational anatomy to quantum control, both of which are discussed in this paper. The problem in its abstract formulation consists of finding an optimal curve $g(t)$ in a Lie group $G$ that generates a curve $q(t) = g(t)Q_0$ in an object manifold $Q$, such that $q(t)$ passes near a set of fixed \emph{target points} at given \emph{node times}. Precise definitions of optimality and proximity will be provided below. 

In the higher-order variational formulation appropriate for such problems, the Euler--Lagrange equations split into a set of Euler--Poincar\'e equations that hold on the open time intervals between the nodes, and a set of \emph{node equations} that describe how to pass from one open interval to the next. The primary purpose of this paper is to develop a new geometric understanding of the node equations in terms of conjugate momenta. Continuing from this, we  develop a numerical algorithm for the trajectory planning problem that respects the geometric properties exhibited by the continuous-time solution.
\subsection{Background and problem formulation}
The problem treated in this paper fits into a classical type of problem in control theory called \emph{trajectory planning}, or \emph{interpolation by variational curves}. The task in this type of problem is to find an optimal curve that interpolates through a given set of points (or configurations) lying in some manifold, specific to the application one has in mind. Such trajectory planning problems are relevant in numerous applications, for example, in aeronautics, robotics, computer-aided design, air traffic control and more recently, computational anatomy. Some trajectory planning applications require the optimal trajectories to possess a certain degree of smoothness. This requirement summons variational principles that depend on higher-order derivatives of the interpolation path, such as \emph{acceleration} (rate of change of velocity) or \emph{jerk} (rate of change of acceleration) etc. Properties of such higher-order variational principles have been widely studied, one of the earlier references being \cite{deLRo1985}, where an intrinsic formulation in terms of higher-order tangent bundles was given. \cite{PR2011} contains a literature overview concerned both with mechanics and field theories, and some further recent developments are given in \cite{Gay-BalmazEtAl2010, Gay-BalmazEtAl2011HOLPHP, CodD2011}. 

In some applications, for example in computational anatomy \cite{Gay-BalmazEtAl2010} and quantum control \cite{BrHoMe2012}, the optimal trajectory is generated via a group action. That is, a curve $g(t)$ in a Lie group $G$ acts on a point $Q_0$ in an \emph{object manifold} and generates a curve $q(t) = g(t) Q_0$, that passes through a sequence of target points at prescribed times. In some instances it is desirable to relax the target constraints in such a way that the optimal curve does not exactly pass through the target points, but still passes \emph{near} them at the prescribed times. This may be achieved by including a \emph{soft constraint} in the cost functional, that is, a term penalizing the discrepancy between the trajectory $q(t)$ and the targets. This leads to cost functionals of the following type,
\begin{equation}\label{Eq_1}
  S = \int \ell(\xi, \dot{\xi}, \ldots, \xi^{(k-1)}) \, \mathrm{d}t + \frac{1}{2\sigma^2} \sum_i d^2(q(t_i), q_i) .
\end{equation}
Here, $\xi^{(j)}$ are the $j$-th time derivatives of a curve $\xi(t)$ in the Lie algebra (the tangent space at the identity element of the Lie group) that integrates to the curve $g(t)$, which in turn produces the trajectory in the object manifold according to $q(t) = g(t) Q_0$. The first part of the cost is the integral over a Lagrangian $\ell$ and is associated with the curve on the group. The second part sums up the squares of the distances $d$ between the curve $q(t)$ and the target points $q_i$, at the prescribed times $t_i$. The \emph{tolerance parameter} $\sigma$ may be adjusted to suitably weight the two parts. In computational anatomy for example, a diffeomorphism group acts by transforming a medical image (or sub-structures thereof, such as  points of interest, fibres, or surfaces).  See \cite{MillerARBE2002} for an overview. In this context the prescribed target images (or sub-structures) may not be diffeomorphically related to the initial one. That is, the initial and target configurations may lie on different group orbits. In general one expects therefore that exact matching may not be possible and works instead with a soft constraint. Two of the authors have recently studied a problem of similar nature in quantum control, see \cite{BrHoMe2012}. There one considers the group of unitary matrices acting on quantum state space, with the goal of finding the optimal experimental manipulation of the system such that the evolution of an initial state passes near a sequence of given target states. The cost functional is directly related to the required amount of change in the experimental apparatus over time. The introduction of a soft constraint is appropriate in this problem even though in this case the group action is transitive, since by increasing the tolerance parameter $\sigma$ optimal trajectories may be achieved at a lower cost. 

In a more general sense such trajectory planning problems can be thought of as \emph{inverse problems}, where the  data points $q_i \in Q$ have been determined experimentally at the times $t_i$ and one seeks the corresponding curve $g(t)$ in configuration space $G$. In this context a natural choice of the tolerance parameter $\sigma$ would be a measure of the uncertainty inherent in the experiment, such as standard deviation. The Lagrangian $\ell$ represents a modeling choice and is specific to the application one has in mind. 

This paper is concerned with trajectory planning problems of the above type. It was found in \cite{Gay-BalmazEtAl2010} that the Euler--Lagrange equations split into higher-order Euler--Poincar\'e equations, which hold on the open intervals between node times, and a set of \emph{node equations} that describe the passage from one open interval to the next. These node equations describe the continuity properties of a set of quantities involving derivatives of various orders of the Lagrangian. As we will see in examples, a natural choice of Lagrangian $\ell$ leads to \emph{Riemannian cubics} and their higher-order generalizations. This class of curves was introduced in \cite{NoHePa1989} and has since been studied in a series of papers including  \cite{CrSL1995, CaSLCr1995, CaSLCr2001, No2004, No2006, splinesanalyse}. Riemannian cubics appear in a variety of applications, for example in the quantum control problem mentioned above, but also in computer graphics, robotics and spacecraft control \cite{PaRa1997, ZeKuCr1998, HuBl2004, TrVi2010}. 
\paragraph{Plan of the paper.}
In Section \ref{Sec_Geometry_of_TP} we shall rederive the Euler--Lagrange equations for the higher-order variational problem by using Lagrange multipliers in a generalization of the symmetry reduced Hamilton--Pontryagin principle of geometric mechanics. In this approach, the derivation of the Euler--Lagrange equations simplifies considerably and a new geometric interpretation of the node equations emerges. Namely, they describe the evolution of Legendre--Ostrogradsky momenta across the nodes, in which the highest-order momentum experiences a discontinuous jump related to the distance between the curve in the object manifold and the target points. The discontinuity can be understood in terms of a momentarily broken symmetry at the node times. However, if the object manifold is isotropic with respect to a subgroup action then a residual symmetry remains. By Noether's theorem, this residual symmetry leads to a conservation law across node times.

In Section \ref{Sec-Applications} we discuss a number of applications, including rigid body splines, macromolecular configurations and quantum splines. Section \ref{Sec_Geom_Dis} is concerned with the numerical solution of the inexact trajectory planning problem. More precisely, Section 4 describes a geometric discretization of the higher-order Hamilton--Pontryagin principle for inexact trajectory planning, similar to the approach given in \cite{BRMa09} for first-order systems. Our main motivation for the development of a geometric integrator is the exact momentum behavior of the discrete solution. This leads in turn to a dimensionality reduction of the search space in the numerical optimization. 
\section{Geometry of the trajectory planning problem}\label{Sec_Geometry_of_TP}
We start with the statement of the problem considered here. One aims at steering from an initial point $Q_0$ in some \emph{object manifold} $Q$ along an optimal trajectory $q(t)$ that evolves via the action of a Lie group $G$. That is, $q(t) = g(t) Q_0$, where the right hand side denotes the action of $g(t)$ on $Q_0$ and the curve $q(t)$ lies in the $G$-orbit of $Q_0$.

The optimality condition is given in terms of a function $\ell: k\mathfrak{g} \to \mathds{R}$ defined on the $k$-fold Cartesian product $k\mathfrak{g}$ of the Lie algebra $\mathfrak{g}$, which measures the cost of the transformation $g(t)$, and a distance function $d: Q\times Q \to \mathds{R}$. As we shall see, the integer $k$ determines the degree of smoothness of solution curves. The optimal curve $q(t)$ is required to pass near prescribed target points $Q_{t_i}$ at prescribed \emph{node times} $t_i$ for $i= 1, \ldots, l$. This is formalized by including a squared distance term $d^2(g(t_i) Q_0, Q_{t_i})$ in the cost functional, for each $i$. Thus, the cost functional $S: \mathcal{C}(\mathfrak{g}) \to \mathds{R}$, where $\mathcal{C}(\mathfrak{g})$ a suitable space of $\mathfrak{g}$-valued curves (see below) is defined by
\begin{align}
  S[\xi]:= \int_0^{t_l} \ell(\xi, \ldots, \xi^{(k-1)})\, \mathrm{d}t + \frac{1}{2\sigma^2} \sum_{i=1}^l d^2(g(t_i)Q_0, Q_{t_i}). \label{Cost_original}
\end{align}
The notation $\xi^{(j)}$ is shorthand for ${ d^{j}\xi}/{dt^{j}}$, the quantity $\sigma$ is a \emph{tolerance parameter}, and the curve $g(t)$ originates at the identity $g(0) = e$ and satisfies $\dot{g} =\frac{d}{d \varepsilon}\big|_{\varepsilon = 0} \exp(\varepsilon \xi)g$. We write $R_g$ for multiplication by $g$ from the right and $TR_g$ for its differential, thus $\dot{g} = TR_g \xi$. Variations are considered in the space $\mathcal{C}(\mathfrak{g})$ consisting of curves $\xi(t): [0, t_l] \to \mathfrak{g}$ whose restrictions to open intervals $(t_{i-1}, t_i)$ for $i = 1, \ldots, l$  are $C^{2k-2}$ and whose $j$-th derivatives $\xi^{(j)}$ are continuous on $[0, t_l]$ for $j = 0, \ldots, k-2$. We also assume that initial values of $\xi^{(j)}(0)$ for $j = 0, \ldots, k-2$, are given.

This type of trajectory planning problem is familiar, for example, from image registration in computational anatomy, where one typically thinks of $Q_0$ as a template shape being deformed by a curve of diffeomorphisms $g(t)$, in turn generated by the time-dependent vector field $\xi(t)$ \cite{Younes2010}.  At times $t_i$ the resulting curve in shape space passes near the given target shapes $Q_{t_i}$, the parameter $\sigma$ determining the proximity of the passage. In this case, the Lie group $G$ of diffeomorphisms is infinite dimensional. However, in the present paper we will restrict ourselves to the case of finite-dimensional Lie groups and object manifolds. A natural finite-dimensional instance for illustrating these ideas arises in quantum control \cite{BrHoMe2012}, where quantum state vectors evolve under the action of the unitary group. The generator curve $\xi(t)$ in this case corresponds to the Hamiltonian operator, which is controlled in experiments. 
\subsection{Euler--Lagrange equations via Lagrange multipliers}
The Euler--Lagrange equations characterize solutions to Hamilton's principle $\delta S = 0$ and were derived in \cite{Gay-BalmazEtAl2010}. The equations split into a set of higher-order Euler--Poincar\'e equations on the open time intervals between the node times and a number of \emph{node equations} describing how the solution evolves across the nodes. In previous formulations of this problem, one must find the variation $\delta g(t_i)$ that is produced by a (time dependent) variation $\delta \xi(t)$. This can be facilitated by taking advantage of Lagrange multipliers in an equivalent variational formulation that we describe now. As this paper demonstrates, the new approach also provides a geometric interpretation of the node equations and furthermore suggests a geometric numerical procedure for the solution of the problem , see Section \ref{Sec_Geom_Dis} below.

The method of Lagrange multipliers involves enlarging the space on which the dynamics happen. We define the cost functional $S$ on some space of curves $\mathcal{C}(G\times k\mathfrak{g} \times k\mathfrak{g}^*)$,
\begin{align}
  S[g, \xi^0, &\ldots, \xi^{k-1}, \mu^0, \ldots, \mu^{k-1}] := \int_0^{t_l}\bigg[ \ell(\xi^0, \xi^1, \ldots, \xi^{k-1}) + \langle\mu^0,TR_{g^{-1}} \dot{g} - \xi^0 \rangle \notag \\ &\quad + \sum_{r = 1}^{k-1}\langle \mu^r, \dot{\xi}^{r-1} - \xi^r\rangle\bigg] \, \mathrm{d}t + \frac{1}{2\sigma^2} \sum_{i = 1}^{l} d^2(g(t_i)Q_0, Q_{t_i}). \label{Cost_LM}
\end{align}
Again some technical assumptions about the space of curves are needed. Namely the curves are $C^1$ when restricted to the  open intervals $(t_{i-1}, t_i)$ for $i = 1, \ldots, l$, and $g, \xi^0, \ldots, \xi^{k-2}$ are continuous on $[0, t_l]$. We also assume $g(0) = e$ and given initial values $\xi^j(0) = \xi_0^j$, for $j = 0, \ldots, k-2$. 

Before we take variations of $S$ it is useful to introduce the \emph{cotangent lift momentum map} $J^Q: T^*Q \to \mathfrak{g}^*$ associated with the action of $G$ on $Q$ (see \cite{MaRa03}, Chapter 11, for more information). This map is defined to satisfy
\begin{align}\label{Moma}
  \langle\alpha_q, \xi_Q(q)\rangle = \langle J^Q(\alpha_q), \xi\rangle, \quad \mbox{for any } \alpha_q \in T^*Q, \, \xi \in \mathfrak{g} \,,
\end{align}
where we have used the notation $\xi_Q(q) := \left.\frac{d}{d \varepsilon}\right|_{\varepsilon = 0} e^{\varepsilon \xi}q\,$ and $\langle\,.\,,.\,\rangle$ for the respective duality pairings.  We also introduce the shorthand $\mathrm{d}_1d(q_1, q_2) \in T_{q_1}^*Q$ to denote the exterior derivative of the distance function $d$ with respect to the first entry, and $\mathrm{d}_1d(t_i):= \mathrm{d}_1d(g(t_i)Q_0, Q_{t_i})$. Integrating by parts and using \eqref{Moma} we obtain
\begin{align}
  \delta S &= \int_0^{t_l}\bigg[ \langle -\dot{\mu}^0 - \operatorname{ad}^*_{\xi^0} \mu^0, \eta\rangle + \sum_{r = 0}^{k-2} \left< \frac{\delta \ell}{\delta \xi^r} - \mu^r - \dot{\mu}^{r+1}, \delta \xi^r\right> + \left< \frac{\delta \ell}{\delta \xi^{k-1}} - \mu^{k-1}, \delta \xi^{k-1}\right> \notag \\
&\quad + \left<\delta \mu^0, TR_{g^{-1}}\dot{g} - \xi^0\right> + \sum_{r=1}^{k-1}\left<\delta \mu^r, \dot{\xi}^{r-1}- \xi^r\right>  \bigg]\, \mathrm{d} t \notag \\
&\quad  + \left<\mu^0(t_l) + \frac{d(t_l)}{\sigma^2}J^Q(\mathrm{d}_1d(t_l)), \eta(t_l)\right> + \sum_{r=0}^{k-2} \left<\mu^{r+1}(t_l), \delta \xi^{r}(t_l)\right> \label{Variations} \\
&\quad +\sum_{s = 1}^{l-1} \bigg[ \left<\mu^0(t_s^-) - \mu^0(t_s^+) + \frac{d(t_s)}{\sigma^2} J^Q(\mathrm{d}_1d(t_s), \eta(t_s)\right>  + \sum_{r=0}^{k-2} \left<\mu^{r+1}(t_s^-) - \mu^{r+1}(t_s^+), \delta \xi^{r}(t_s)\right> \bigg]  \notag \\
&\quad - \left<\mu^0(0) - \frac{d(0)}{\sigma^2}J^Q(\mathrm{d}_1d(0)), \eta(0)\right> - \sum_{r=0}^{k-2} \left<\mu^{r+1}(0), \delta \xi^{r}(0)\right>, \notag
\end{align}
where we set $\eta := TR_{g^{-1}}\delta g$ and defined $\mu^r(t_s^-) := \operatorname{lim}_{t\uparrow t_s}\mu^r(t)$ as well as  $\mu^r(t_s^+) := \operatorname{lim}_{t\downarrow t_s}\mu^r(t)$. The last line above could have been omitted since by assumption $\eta(0) =0$ and $\delta \xi^j =0$ for $j = 0, \ldots, k-2$. 

We can now read off the Euler--Lagrange equations. On the one hand, for $t$ in any of the open intervals $(t_i, t_{i+1})$, $i = 0, \ldots, l-1$, we have
\begin{align}
  &\dot{\mu}^0 + \operatorname{ad}^*_{\xi^0} \mu^0 = 0, \label{open_1}\\
& TR_{g^{-1}}\dot{g} - \xi^0 = 0, \label{open_2} \\
& \dot{\xi}^{r-1} - \xi^r = 0, \qquad (r = 1, \ldots, k-1) \label{open_3}\\
  &\dot{\mu}^r + \mu^{r-1} - \frac{\delta \ell}{\delta \xi^{r-1}} = 0, \qquad (r = 1, \ldots, k-1) \label{open_4} \\
&\mu^{k-1} - \frac{\delta \ell}{\delta \xi^{k-1}} = 0.   \label{open_5}
\end{align}
On the other hand, the \emph{node equations} are given by
\begin{align}
&\mu^0(t_s^-) - \mu^0(t_s^+) + \frac{d(t_s)}{\sigma^2} J^Q(\mathrm{d}_1d(t_s)) = 0, \qquad (s = 1, \ldots, l-1) \label{node_1} \\
  &\mu^{r}(t_s^-) - \mu^{r}(t_s^+), \qquad (r = 1, \ldots, k-1;\, s = 1, \ldots, l-1) \label{node_2} \\
  &\mu^0(t_l) + \frac{d(t_l)}{\sigma^2}J^Q(\mathrm{d}_1d(t_l)) = 0, \label{node_3}\\
  &\mu^r(t_l) = 0. \qquad (r = 1, \ldots, k-1) \label{node_4}  
\end{align}
\begin{remark}\label{Rem_1}
There are $4$ versions of the action functional, which are all relevant in applications. The one above can be called the \emph{left-action, right-reduction} version since $g(t)$ acts on $Q_0$ from the left, while $\xi^0$ is the right-reduced velocity $\xi^0 = TR_{g^{-1}}\dot{g}$  (see Section \ref{Sec_EP_eqns} below, for more details on reduced variables). There are the following three other cases.

\begin{enumerate}[(1)]
\item
The \emph{right-action, right-reduction} case with action functional
\begin{align}
  S[g, \xi^0, &\ldots, \xi^{k-1}, \mu^0, \ldots, \mu^{k-1}] := \int_0^{t_l}\bigg[ \ell(\xi^0, \xi^1, \ldots, \xi^{k-1}) + \langle\mu^0, TR_{g^{-1}}\dot{g} - \xi^0 \rangle \notag \\ &\quad + \sum_{r = 1}^{k-1}\langle \mu^r, \dot{\xi}^{r-1} - \xi^r\rangle\bigg] \, \mathrm{d}t + \frac{1}{2\sigma^2} \sum_{i = 1}^{l} d^2(g(t_i)^{-1}Q_0, Q_{t_i}). \notag
\end{align}
The variation of the penalty term in this case changes according to
\begin{align}
  \frac{1}{2} \delta d^2(g(t_i)^{-1}Q_0, Q_{t_i}) = -d(t_i) \left<\operatorname{Ad}^*_{g(t_i)^{-1}} J({\rm d}_1 d(t_i)), \eta(t_i)\right>. \notag
\end{align}
This means that \eqref{node_1} and \eqref{node_3} are replaced by
\begin{align}
  &\mu^0(t_s^-) - \mu^0(t_s^+) - \frac{d(t_s)}{\sigma^2} \operatorname{Ad}^*_{g(t_s)^{-1}} J^Q(\mathrm{d}_1d(t_s)) = 0, \qquad (s = 1, \ldots, l-1)\notag \\
 &\mu^0(t_l) -  \frac{d(t_l)}{\sigma^2}\operatorname{Ad}_{g(t_l)^{-1}}^*J^Q(\mathrm{d}_1d(t_l)) = 0.\notag
\end{align}
\item
In the \emph{right-action, left-reduction} case, the action functional is
\begin{align}
   S[G, \Xi^0, &\ldots, \Xi^{k-1}, m^0, \ldots, m^{k-1}] := \int_0^{t_l}\bigg[ l(\Xi^0, \Xi^1, \ldots, \Xi^{k-1}) + \langle m^0, TL_{G^{-1}}\dot{G} - \Xi^0 \rangle \notag \\ &\quad + \sum_{r = 1}^{k-1}\langle m^r, \dot{\Xi}^{r-1} - \Xi^r\rangle\bigg] \, \mathrm{d}t + \frac{1}{2\sigma^2} \sum_{i = 1}^{l} d^2(G(t_i)^{-1}Q_0, Q_{t_i}), \notag
\end{align}
where we wrote $L_G$ for multiplication by $G$ from the left and $TL_G$ for its differential. However, this is equivalent to the left-action, right-reduction case by identifying
\begin{align}
  G = g^{-1},\quad \Xi^0 = -\xi^0, &\ldots, \Xi^{k-1} = -\xi^{k-1}, \quad m^0 = -\mu^0, \ldots, m^{k-1} = -\mu^{k-1}
\end{align}
and setting $\ell = l \circ \kappa$, where $\kappa: k\mathfrak{g} \to k\mathfrak{g}$ is multiplication by $-1$. 
\item
By the same token the \emph{left-action, left-reduction} case with action functional
\begin{align}
   S[G, \Xi^0, &\ldots, \Xi^{k-1}, m^0, \ldots, m^{k-1}] := \int_0^{t_l}\bigg[ l(\Xi^0, \Xi^1, \ldots, \Xi^{k-1}) + \langle m^0, TL_{G^{-1}}\dot{G} - \Xi^0 \rangle \notag \\ &\quad + \sum_{r = 1}^{k-1}\langle m^r, \dot{\Xi}^{r-1} - \Xi^r\rangle\bigg] \, \mathrm{d}t + \frac{1}{2\sigma^2} \sum_{i = 1}^{l} d^2(G(t_i)Q_0, Q_{t_i}). \notag
\end{align}
can be mapped to the right-action, right-reduction case.
\end{enumerate}
In the analysis that follows we will largely restrict ourselves to the left-action, right-reduction case. Anything we say can be transferred to the remaining three cases by applying the modifications listed above.
\end{remark}
\subsection{Euler--Poincar\'e equations} \label{Sec_EP_eqns} From \eqref{open_1}--\eqref{open_5} it follows that  on open intervals $(t_i, t_{i+1})$, $i = 0, \ldots, l-1$,
\begin{align}
  \left(\frac{d}{dt} + \operatorname{ad}^*_{\xi^0}\right) \sum_{j=0}^{k-1} (-1)^j \frac{d^j}{dt^j} \frac{\delta \ell}{\delta \xi^j} = 0. \label{EP_TP}
\end{align}
This is a \emph{$k$-th order Euler--Poincar\'e equation} for a system that exhibits right-invariance. This type of equation was derived in \cite{Gay-BalmazEtAl2010} from a variational perspective and in \cite{CodD2011} from the Hamiltonian one. We now explain its appearance in the inexact trajectory planning problem from the viewpoint of invariant Lagrangians, starting with some necessary definitions that can be found in \cite{CeMaRa2001} and \cite{Gay-BalmazEtAl2011HOLPHP}.

The \emph{$k$-th order tangent bundle} $ \tau_G^{(k)}: T^{(k)}G \to G$ is defined as a set of equivalence classes of curves as follows: Two curves $g_i(t) \in G, i = 1, 2$, are \emph{equivalent}, if and only if their time derivatives at $t=0$ up to order $k$ coincide in any local chart. That is, $g_1^{(j)}(0) = g_2^{(j)}(0)$, for $0 \leq j \leq k$. The equivalence class of a curve $g(t)$ is denoted by $[g]_{g(0)}^{(k)}$, or formally as $(g(0), \dot{g}(0), \ldots, g^{(k)}(0))$. The set of all equivalence classes of curves emanating from $g_0 \in G$ is written as $T_{g_0}^{(k)}G$. Then $T^{(k)}G := \bigcup_{g \in G} T_{g}^{(k)}G$ is a fibre bundle over $G$ with projection map $\tau^{(k)}_G: [g]_{g(0)}^{(k)} \mapsto g(0)$. Note that a curve $g(t)$ defines, at each time $t$ in its domain, an element $[g]_{g(t)}^{(k)}:= [h]_{h(0)}^{(k)}$ by setting $h(s):= g(t+s)$.

It is convenient to represent $T^{(k)}G$ using the \emph{trivialization map} that makes use of the right group multiplication (analogous constructions exist using the left multiplication map). Let $g(t)$ be a representative of $[g]_{g(0)}^{(k)}$ and define $\xi(t) :=TR_{g^{-1}} \dot{g}(t)$ . The trivialization map $T^{(k)}G \to G \times k\mathfrak{g}$ is given by
\begin{align}
  \alpha_k: [g]_{g(0)}^{(k)} \mapsto \big(g(0), \xi(0), \dot{\xi(0)}, \ldots, \xi^{(k-1)}(0)\big). \notag
\end{align}
The \emph{reduction map} $\beta_k: T^{(k)}G \to k\mathfrak{g}$ is obtained by omitting the first entry.

It is a well known (see, for example \cite{MaRa03}, Chapter 13) that the first order $(k=1)$ Euler--Poincar\'e equation appears when the Euler--Lagrange equation for a Lagrangian $TG \to \mathds{R}$ with symmetry is written in terms of the reduced velocity vector $\xi(t) =TR_{g^{-1}(t)} \dot{g}(t)$ . The higher-order Euler--Poincar\'e equation appears in a similar fashion when computing the Euler--Lagrange equations for an invariant $k$-th order Lagrangian $L: T^{(k)}G \to \mathds{R}$. More precisely, $L$ is called \emph{right-invariant} if $L\big([g]_{g(0)}^{(k)}\big) = L\big([gh]_{g(0)h}^{(k)}\big)$. The definition of \emph{left-invariance} follows analogously. Let $L$ be a right-invariant Lagrangian and consider Hamilton's principle, $\delta \mathcal{J} = 0$, for
\begin{align} \label{Standard_action}
 \mathcal{J}[g] =  \int_a^b L(g(t), \dot{g}(t), \ldots, g^{(k)}(t)) \, {\rm d}t, 
\end{align}
where variations are taken with respect to fixed end points up to order $k-1$, that is, $\delta g^{(j)}(a) = \delta g^{(j)}(b) = 0$, for $j = 0, \ldots k-1$, in any local chart. The Euler--Lagrange equations can be written in terms of the right-reduced velocity vector, which leads to the \emph{$k$-th order Euler--Poincar\'e equations} \cite{Gay-BalmazEtAl2010}
  \begin{align}
  \left(\frac{d}{dt} + \operatorname{ad}^*_{\xi}\right) \sum_{j=0}^{k-1} (-1)^j \frac{d^j}{dt^j} \frac{\delta \ell}{\delta \xi^{(j)}} = 0, \label{EP_standard}
\end{align}
with \emph{reduced Lagrangian} $\ell: k\mathfrak{g} \to \mathds{R}$,
\[ \ell\big(\xi(0), \dot{\xi}(0), \ldots, \xi^{(k-1)}(0)\big) = L \circ \alpha_k^{-1}\big(e, \xi(0), \dot{\xi}(0), \ldots, \xi^{(k-1)}(0)\big).
\]
It is now straightforward to see from the viewpoint of invariant Lagrangians why the Euler--Poincar\'e equation \eqref{EP_TP} must characterize optimal curves in the trajectory planning problem on open time intervals. Indeed, fix $0 \leq i \leq l$ and suppose $\xi(t)$ is a local extremum of \eqref{Cost_original} with integral curve $g(t)$. If $\xi|_{(t_i, t_{i+1})}$ is not a solution to the $k$-th order Euler--Poincar\'e equation  one can find a variation $g_{\varepsilon}(t)$ keeping fixed $g(t_i)$ and $g(t_{i+1})$, such that $\delta \int_{t_i}^{t_{i+1}} L(g, \dot{g}, \ldots, g^{(k)}) \, \mathrm{d} t \neq 0$ with Lagrangian $L := \ell\circ \beta_k$. By consequence,  $\xi_{\varepsilon}(t) :=TR_{ g_{\varepsilon}^{-1}(t)} \dot{g}_{\varepsilon}(t)$ is a variation of $\xi$ such that $\delta S \neq 0$ for $S$ as in \eqref{Cost_original}, a contradiction.
\subsection{Geometry of multipliers} \label{HOHP_and_geometry_of_mult} Furthermore we observe from \eqref{open_1}--\eqref{open_5} that  on open intervals $(t_i, t_{i+1})$, $i = 0, \ldots, l-1$,
\begin{align}
  \mu^r = \sum_{j=0}^{k-r-1} (-1)^j \frac{d^j}{dt^j} \frac{\delta \ell}{\delta \xi^{r+j}} \qquad (r = 0, \ldots, k-1). \label{Leg-Ost_TP}
\end{align}
In order to discuss the geometric meaning of these identities, we recall from above that the trajectory planning problem \eqref{Cost_LM} on open intervals reduces to a problem of the type $\delta \mathcal{J} = 0$ with $\mathcal{J}$ of the form \eqref{Standard_action}. In particular, equations \eqref{open_1}--\eqref{open_5} and therefore \eqref{Leg-Ost_TP} are obtained by taking suitably constrained variations of
\begin{align}
  \mathcal{J}[g, \xi^0,& \ldots, \xi^{k-1}, \mu^0, \ldots, \mu^{k-1}] := \int_a^{b}\bigg[ \ell(\xi^0, \xi^1, \ldots, \xi^{k-1}) + \langle\mu^0,TR_{g^{-1}} \dot{g} - \xi^0 \rangle \notag \\ &\quad + \sum_{r = 1}^{k-1}\langle \mu^r, \dot{\xi}^{r-1} - \xi^r\rangle\bigg] \mathrm{d}t \label{HP_principle}
\end{align}
This variational principle is a higher-order generalization of the \emph{reduced Hamilton--Pontryagin principle} of first order mechanics. In first order mechanics, this principle provides a unified treatment of the Lagrangian and Hamiltonian descriptions of invariant mechanical systems on Lie groups (see \cite{YoMa06} for a detailed discussion\footnote{There is a close connection between the Hamilton--Pontryagin principle and Dirac structures, an aspect we do not enter into in the present paper. See \cite{YoMa06} for more information.}). In particular, the \emph{Legendre transform} connecting the two descriptions is revealed by the variational calculus. This remains true for higher-order mechanics. Indeed, \eqref{Leg-Ost_TP} can be recognized to be the reduced Legendre transform that appears in \cite{Gay-BalmazEtAl2010,Gay-BalmazEtAl2011HOLPHP}. While we found \eqref{Leg-Ost_TP} from a variational approach, these references take as starting point \cite{deLRo1985}, where a coordinate free description of the higher-order Legendre transform on manifolds was given. We briefly review this approach here.  

The Legendre transform of higher-order mechanics, given in \cite{deLRo1985}, is a map $\mbox{Leg}: T^{(2k-1)}G \to T^*(T^{(k-1)})G$. If Leg is a diffeomorphism (that is, $L$ is \emph{hyperregular}) it connects the Lagrangian and Hamiltonian descriptions  just as in the first order case. With respect to the right-trivializations
\begin{align}
  T^{(2k-1)}G \cong G \times (2k-2)\mathfrak{g}, \qquad T^*(T^{(k-1)})\cong G\times (k-2)\mathfrak{g} \times (k-1)\mathfrak{g}^* \label{Trivs}
\end{align}
it is given as \cite{Gay-BalmazEtAl2011HOLPHP}
\begin{align}
  \mbox{Leg:}& \,(g, \xi^0, \ldots, \xi^{2k-2}) \mapsto (g, \xi^0, \ldots, \xi^{k-2}, \mu^0, \ldots, \mu^{k-1}), \notag \\
&\mbox{where} \quad \mu^r =  \sum_{j=0}^{k-r-1} (-1)^j \frac{d^j}{dt^j} \frac{\delta \ell}{\delta \xi^{r+j}} \quad (r = 0, \ldots k-1).  \notag
\end{align}
The same equations were seen in \eqref{Leg-Ost_TP} to emerge from the Hamilton--Pontryagin principle \eqref{HP_principle}.  This means that, as for first order, the higher-order Hamilton--Pontryagin principle contains both Lagrangian and Hamiltonian descriptions of higher-order mechanics and provides a unified framework for both views. To obtain the Lagrangian description one may eliminate $\mu^0, \ldots, \mu^{k-1}$ from \eqref{open_1}--\eqref{open_5} using \eqref{Leg-Ost_TP}. The resulting equations are the trivialized flow equations of the \emph{Lagrangian vector field}, an element of $\mathfrak{X}(T^{(2k-1)}G)$ \cite{deLRo1985}. On the other hand, if \eqref{open_5} can be solved for $\xi^{k-1}$ (this is the case, for example, when $L$ is hyperregular) then \eqref{open_1}--\eqref{open_4} are the trivialized flow equations of the \emph{Hamiltonian vector field} $X_H \in \mathfrak{X}(T^*(T^{(k-1)})$, which solves \cite{deLRo1985}
\begin{align}
  i_{X_H}\omega = \mathrm{d}H,  \label{HamVF}
\end{align}
where $\omega$ is the canonical symplectic form on $T^*(T^{(k-1)}G)$ and $H: T^*(T^{(k-1)}G) \to \mathds{R}$ is given as
\begin{align}
  H(g, \xi^0, \ldots, \xi^{k-2}, \mu^0, \ldots \mu^{k-1}) = \sum_{r = 0}^{k-1} \left<\mu^r, \xi^r\right> - \ell(\xi^0, \ldots, \xi^{k-1}) \label{Hamiltonian}
\end{align}
 with respect to the trivialization \eqref{Trivs}. By consequence of \eqref{HamVF} the flow map $F_t: T^*(T^{(k-1)}) \to T^*(T^{(k-1)})$ of the Hamiltonian vector field preserves the symplectic form $\omega$. 

For later reference we point out how this can be seen alternatively from the Hamilton--Pontryagin principle. This is a generalization to higher order of a standard argument (see for example \cite{BRMa09}, Section 3, for the first order case). If we omit end point constraints on the variations of $\mathcal{J}$ in \eqref{HP_principle}, the integration by parts contributes boundary terms to $\delta \mathcal{J}$ (cf. \eqref{Variations}),
\begin{align}
 \delta \mathcal{J} = \int_a^b \cdots\, \mathrm{d}t  + \left. \left<\mu^0, TR_{g^{-1}}\delta g\right>\right|_a^b + \sum_{r = 0}^{k-2} \left.\left<\mu^{r+1}, \delta \xi^r\right>\right|_{a}^b = \int_a^b \cdots\, \mathrm{d}t +  \left.\theta(\delta x)\right|_a^b, \label{Symplectic_arg}
\end{align}
where $\theta$ in the second equality is the canonical one-form on $T^*(T^{(k-1)}G)$ and we defined $\delta x(t)$ to be the curve in $TT^*(T^{(k-1)}G)$ whose trivialization corresponds to the variations $(TR_{g^{-1}}\delta g , \delta \xi^0, \ldots, \delta \xi^{(k-2)}, \delta \mu^0, \ldots, \delta \mu^{(k-1)})$. If we restrict the variations to solution curves of \eqref{open_1}--\eqref{open_5}, we may just as well express $\mathcal{J}$ as a function of initial conditions $\mathcal{J}_{\mbox{initial}}: T^*(T^{(k-1)}G) \to \mathds{R}$. The integral part of \eqref{Symplectic_arg} then vanishes and we obtain 
\[
\delta \mathcal{J}  = \mathrm{d}\mathcal{J}_{\mbox{initial}} (\delta x(a))= (F_{b-a}^*\theta - \theta) (\delta x(a)).
\]
Therefore, $F_{b-a}^*\theta = \theta$, and taking an exterior derivative yields the desired identity $F_{b-a}^*\omega = \omega$.
\subsection{Momentum conservation and Noether's theorem} \label{Section_MCNT} Another important point to notice from \eqref{open_1}--\eqref{open_5} is that $\operatorname{Ad}_{g}^*\mu^0$ is a conserved quantity. Here $\operatorname{Ad}^*$ is the map dual to $\operatorname{Ad}: G \times \mathfrak{g} \to \mathfrak{g}$ given by $(g, \xi) \mapsto \operatorname{Ad}_g\xi:= TL_g TR_{g^{-1}} \xi$. Indeed by \eqref{open_1}
\begin{align}
  \frac{d}{dt} \operatorname{Ad}_g^*\mu^0 = \operatorname{Ad}^*_g (\dot{\mu}^0 + \operatorname{ad}^*_{\xi^0} \mu^0) = 0. \label{Conservation}
\end{align}
In the context of first order Euler--Poincar\'e equations a similar momentum conservation is due to the invariance of the Lagrangian with respect to group multiplication operations. This is an instance of Noether's theorem, which roughly speaking guarantees that the momentum map associated with the action of a symmetry group is preserved (see for example \cite{MaRa03}, Chapter 11). We now show that the situation is similar for \eqref{Conservation}. The right action $R$ of $G$ on itself,
\begin{align}
  R: G \times G, \quad (h, g) \mapsto R_g(h) = hg, \notag
\end{align}
can be lifted to an action on $T^{(k-1)}G$,
\begin{align}
  T^{(k-1)}R: T^{(k-1)}G \times G \to T^{(k-1)}G, \quad \left([h]_{h(0)}^{(k-1)}, g\right) \mapsto T^{(k-1)}R_g\left([h]_{h(0)}^{(k-1)}\right) =  [hg]_{h(0)g}^{(k-1)}. \notag
\end{align}
This action can be lifted to its so-called \emph{cotangent lifted action} (\cite{MaRa03}, Section 12.1)
\begin{align}
  T^*T^{(k-1)}R: T^*(T^{(k-1)}G) \times G \to T^*(T^{(k-1)}G), \notag
\end{align}
given in trivialized form as
\begin{align}
   T^*T^{(k-1)}R_g (h, \xi^0, \ldots \xi^{k-2}, \mu^0, \ldots \mu^{k-1}) = (hg^{-1}, \xi^0, \ldots, \xi^{k-2}, \mu^0, \ldots \mu^{k-1}). \notag
\end{align}
It is apparent that the Hamiltonian \eqref{Hamiltonian} is symmetric with respect to this group action. By Noether's theorem the associated momentum map is conserved. 

What is this momentum map? By appealing to standard formulas (\cite{MaRa03}, Section 12.1) we find that the momentum map $J: T^*(T^{(k-1)}G) \to \mathfrak{g}^*$ is
\begin{align}
  J(g, \xi^0, \ldots \xi^{k-2}, \mu^0, \ldots \mu^{k-1}) = \operatorname{Ad}^*_g \mu^0, \notag
\end{align}
with respect to the trivialization \eqref{Trivs}. The conservation law observed in \eqref{Conservation} thus arises from the right-invariance of the Lagrangian (respectively, the Hamiltonian) via Noether's theorem. 

The conservation law can also be obtained from a variational perspective. This is well known in first order mechanics, and it is also the case in higher-order mechanics. We take a solution of \eqref{open_1}--\eqref{open_5} on the time interval $[a, b]$ and vary it according to $\delta g = TL_g \nu$ for $\nu \in \mathfrak{g}$. For $\mathcal{J}$ as in \eqref{HP_principle} we have (cf. \eqref{Symplectic_arg})
\begin{align}
  \delta \mathcal{J} = 0 = \left.\left<\mu^0, TR_{g^{-1}}TL_g\nu\right>\right|_a^b = \left.\left<\operatorname{Ad}^*_g\mu^0, \nu\right>\right|_a^b. \label{Noether_var}
\end{align}
The same argument holds after replacing the upper boundary $b$ by any $b' \in [a, b]$. Since $\nu$ was arbitrary we conclude that $\operatorname{Ad}^*_g \mu^0$ is conserved along a solution of \eqref{open_1}--\eqref{open_5}.
\subsection{Node equations} The remarks above concerned equations \eqref{open_1}--\eqref{open_5} on the open time intervals between nodes. We now come to the node equations \eqref{node_1}--\eqref{node_4}. These specify the evolution  across node times of the Lagrange multipliers $\mu^r$, which we interpreted above as the reduced \emph{Legendre momenta} of the system. More specifically, the momenta $\mu^r, \, r = 1, \ldots, k-1$ are continuous on $[0, t_l]$, while the $0$-th momentum $\mu^0$ experiences jump discontinuities at the nodes. If the Lagrangian $\ell$ is hyperregular we can conclude that $g \in C^{2k-2}([0, t_l])$, that is, $g$ is $(2k-2)$ times continuously differentiable on $[0, t_l]$. Furthermore, the node equations specify terminal values for the curves $\mu^r, \, r = 1, \ldots, k-1$. 

For $a, \, b \in \mathds{R}$ define $1_{a\leq b}$ to be equal to $1$ if $a\leq b$ and $0$ otherwise. We can now prove the following theorem.
\begin{theorem} \label{Theorem_m0}
  For $t$ in any of the open time intervals $(t_s, t_{s+1})$ as well as for $t \in \{0, t_l\}$, 
  \begin{align}
    \mu^0(t) = -\frac{1}{\sigma^2} \operatorname{Ad}^*_{g(t)^{-1}} \left(\sum_{s = 1}^{l} 1_{t\leq t_s} d(t_s)\operatorname{Ad}^*_{g_{t_s}} J^Q(\mathrm{d}_1d(t_s))\right). \label{Thm_eq}
  \end{align}
\end{theorem}
\begin{proof}
  At final time $t = t_l$ \eqref{Thm_eq} clearly holds because of \eqref{node_3}. Since $\operatorname{Ad}^*_g\mu^0$ is conserved on open intervals it follows that for $t \in (t_s, t_{s+1})$,
\begin{align}
  \mu^0(t) = \operatorname{Ad}^*_{g(t)^{-1}}\operatorname{Ad}^*_{g(t_{s+1})}\mu^0(t_{s+1}^-). \notag
\end{align}
We can now obtain \eqref{Thm_eq} by induction over the open time intervals, noting that at each node $t = t_s$ a term $-\frac{d(t_s)}{\sigma^2} J^Q(\mathrm{d}_1 d(t_s))$ gets added on.
\end{proof}
In order to formulate the following corollary we use the notation $\mathfrak{g}_q$ for any point $q \in Q$ to denote the Lie algebra of the \emph{isotropy subgroup} of that point,  $G_q := \left\{ g \in G \big| gq = q\right\}$. In particular $\rho_Q(q) = 0$ for any $\rho \in \mathfrak{g}_q$.
\begin{corollary}\label{Corollary}
   For a solution of \eqref{open_1}--\eqref{node_4} we have, for $t$  in any of the open time intervals $(t_s, t_{s+1})$  as well as for $t \in \{0, t_l\}$,
  \begin{align}
    \left<\mu^0(t), \rho\right> = 0 \qquad \mbox{for all }\, \rho \in \mathfrak{g}_{q(t)}. \label{Thm}
  \end{align}
\end{corollary}
In Section \ref{Sec_Geom_Dis} we will develop a geometric algorithm that inherits an exact version of this Corollary. This implies that the numerical search for the optimal initial value of  $\mu^0$ can be  restricted to the subspace of $\mathfrak{g}^*$ that annihilates $\mathfrak{g}_{Q_0}$.
\begin{proof} For $t$ and $\rho$ as in the statement of the corollary  it follows from Theorem \ref{Theorem_m0} that
  \begin{align}
     \left<\mu^0(t), \rho\right> &=  -\frac{1}{\sigma^2} \sum_{s = 1}^{l} 1_{t\leq t_s} \left< d(t_s) J^Q(\mathrm{d}_1d(t_s)),  \operatorname{Ad}_{g_{t_s}} \operatorname{Ad}_{g(t)^{-1}}\rho\right> \notag \\
&=  -\frac{1}{\sigma^2} \sum_{s = 1}^{l} 1_{t \leq t_s} \left<d(t_s)\, \mathrm{d}_1d(t_s),  \left(\operatorname{Ad}_{g_{t_s}} \operatorname{Ad}_{g(t)^{-1}}\rho\right)_Q(q(t_s))\right> = 0, \notag 
  \end{align}
where we used \eqref{Moma} for the second equality and noted that $\operatorname{Ad}_{g_{t_s}}  \operatorname{Ad}_{g(t)^{-1}}\rho \in \mathfrak{g}_{q(t_s)}$ for the third.
\end{proof}
\subsection{Residual conservation law after partial symmetry breaking} A physically intuitive perspective on Corollary \ref{Corollary} is to understand it as a residual conservation law after partial symmetry breaking. We can see the sum in \eqref{Cost_LM} as the integral over a time-dependent \emph{potential function} $V: [0, t_l] \times G \to \mathds{R}$ given by
\begin{align}
  V(t, g) =  \frac{1}{2\sigma^2} \sum_{i = 1}^{l} \delta(t - t_i)\, d^2(g(t)Q_0, Q_{t_i}) \label{potential}
\end{align}
This potential produces instantaneous singular forces at node times $t_i$ which impart the jump discontinuities on the otherwise conserved momentum $ J = \operatorname{Ad}_g^*\mu^0$, 
\begin{align}
 J(t_s^+) =  J(t_s^-) + \frac{d(t_s)}{\sigma^2} \operatorname{Ad}^*_{g(t_s)}J^Q(\mathrm{d}_1d(t_s)). \label{Mom_jump}
\end{align}
 This is because the presence of this potential breaks the $G$-invariance of the variational problem, however a residual symmetry remains. Clearly, multiplication of $g$ from the right by an element $h \in G_{Q_0}$ leaves $V$ invariant.  An adaptation of the argument surrounding equation \eqref{Noether_var}, restricting $\nu$ to the subspace $\mathfrak{g}_{Q_0} \subset \mathfrak{g}$, then leads to
\begin{align}
  0 = \left.\left<\operatorname{Ad}_{g}^*\mu^0, \nu\right>\right|_0^t. \notag
\end{align}
for any $t \in [0, t_l]$. Moreover, \eqref{node_3} guarantees that $\left<\operatorname{Ad}_{g(t_l)}^*\mu^0(t_l), \nu\right> = 0$. Therefore, $\left<\operatorname{Ad}_{g(t)}^*\mu^0(t), \nu\right> = 0$ for any $t \in [0, t_l]$ and $\nu \in \mathfrak{g}_{Q_0}$, which is equivalent to Corollary \ref{Corollary}.
\section{Applications}\label{Sec-Applications}
In this section we discuss a number of examples that summon the inexact trajectory planning problem. We start with a discussion of Riemannian cubics on general Riemannian manifolds and specifically on Lie groups with invariant metrics, since this type of curve appears for a certain natural choice of Lagrangian. We then treat the rigid body, finite-dimensional quantum systems and molecular strands.
\subsection{Riemannian cubics}
Let $(M, \gamma)$ be a Riemannian manifold with metric $\gamma$, whose norm we denote by $\left\|. \right\|_\gamma$. The notion of straight lines in Euclidean space generalizes to geodesics in $M$. These are curves $x(t) \in M$ that satisfy $D_t\dot{x} = 0$. Here we introduce the notation $D_t\dot{x} = \nabla_{\dot{x}}\dot{x}$, where $\nabla$ is the Levi--Civita connection for $\gamma$. In local coordinates the geodesic equation is given by
\begin{align}
  \ddot{x}^k + \Gamma_{ij}^k\dot{x}^i\dot{x}^j = 0, \notag
\end{align}
where $\Gamma_{ij}^k$ are the Christoffel symbols of the Levi--Civita connection. We define the vector bundle isomorphism over the identity, $\flat: TM \to T^*M$, which maps $v \in T_xM$ to $v^\flat = \gamma(x)(v_x, \cdot)$. The inverse map is denoted by $\sharp$.

The physical interpretation of straight lines in Euclidean space follows from Newton's second law $\ddot{x} = F$. Here $F$ is an external force acting on a particle of unit mass that follows the trajectory $x(t)$. By analogy one can often understand $(D_t\dot{x})^\flat$ as a generalized force acting on a physical system whose time evolution is represented by a curve $x(t)$ in configuration space $M$. Consider the Lagrangian $L: T^{(2)}M \to \mathds{R}$,
\begin{align}\label{Lag_cubics}
  L(x, \dot{x}, \ddot{x}) = \frac{1}{2} \left\|D_t\dot{x}\right\|_\gamma^2, 
\end{align}
whose corresponding action functional  $S = \int_a^b L \, {\rm d}t$ measures the square of the $L^2$-norm of the external force. Hamilton's principle $\delta S = 0$ for variations with fixed boundary velocities $\dot{x}(a)$ and $\dot{x}(b)$ leads to the Euler--Lagrange equation \cite{NoHePa1989}
\begin{align}
  D_t^3 \dot{x} + R(D_t\dot{x}, \dot{x})\dot{x} = 0\,, \notag
\end{align}
where $R$ is the curvature tensor, $R(X, Y)Z := \nabla_X\nabla_YZ - \nabla_Y \nabla_X Z - \nabla_{[X, Y]}Z$ for any vector fields $X, Y, Z \in \mathfrak{X}(M)$. Solutions to this equation are called \emph{Riemannian cubics} and were introduced in \cite{NoHePa1989}. 

When the manifold is a Lie group $M = G$ we call a Riemannian metric \emph{right (or left) invariant} if $\gamma(g)(v_g, w_g) = \gamma(gh)(TR_h v_g, TR_h w_g)$, or  $\gamma(g)(v_g, w_g) = \gamma(hg)(TL_h v_g, TL_h w_g)$ respectively, for all $g,\, h \in G$ and $v_g,\, w_g \in T_gG$. If the metric is right (or left) invariant, then the Lagrangian \eqref{Lag_cubics} can be reduced to a function $\ell: 2\mathfrak{g} \to \mathds{R}$ \cite{Gay-BalmazEtAl2010}
\begin{align}
  \ell(\xi^0, \xi^1) = \frac{1}{2} \left\|\xi^1 \pm \operatorname{ad}^\dagger_{\xi ^0}\xi^0\right\|_\gamma^2, \label{Red_Lag_cubics}
\end{align}
where we introduced the operation $\operatorname{ad}^\dagger_{\xi}\rho = (\operatorname{ad}^*_\xi \rho^\flat)^\sharp$. This can be seen from the fact (e.g., \cite{Gay-BalmazEtAl2010}) that for a curve $g(t)$, whose right (or left) reduced velocity is given by $\xi(t) \in \mathfrak{g}$,
\begin{align}
  D_t\dot{g} = \left(\dot{\xi} + \operatorname{ad}^\dagger_\xi \xi\right)g, \quad \mbox{or} \quad  D_t\dot{g} = g\left(\dot{\xi} - \operatorname{ad}^\dagger_\xi \xi\right). \notag
\end{align}
In the following we will discuss a number of physical systems whose configuration spaces are Lie groups and whose equation of motion in the absence of external forces is given by $D_t\dot{g} = 0$ for some right (or left) invariant Riemannian metric. 
The Lagrangian \eqref{Red_Lag_cubics} is then one natural choice for $\ell$ in the inexact trajectory planning problem since optimal curves minimize (in the $L^2$ sense) the amount of external forcing necessary to achieve them.
It is clear from equations of motion \eqref{open_1}--\eqref{node_2} that the solution  $g(t)$ is a Riemannian cubic on open intervals, and twice continuously differentiable on the whole time interval $[0, t_l]$. Such curves are called \emph{Riemannian cubic splines}.
\begin{remark}\label{TrViRemark}
Without going into the mathematical details we point to a probabilistic interpretation of Riemannian cubics. See also \cite{TrVi2010}, where a closely related idea is discussed in the context of stochastic modeling of biological growth. 
Let $G$ be a Lie group with right-invariant metric $\gamma$ and let $e_i$, $i =1, \ldots, d$ be an orthonormal basis of the $d$-dimensional Lie algebra $\mathfrak{g}$. Consider a curve $g(t) \in G$, whose right-reduced velocity $\xi =TR_{g^{-1}} \dot{g}$ satisfies the following Ito stochastic differential equation
\begin{align}\label{SDE}
  {\rm d}\xi = - \operatorname{ad}^\dagger_{\xi}\xi \, {\rm d}t + \sigma_W \sum_{i = 1}^d {\rm d}W_i e_i, 
\end{align}
where $W_i$, $i = 1, \ldots, d$, are  independent Brownian motions (see, for example, \cite{Oks2003} Section 2.2, for a definition) and $\sigma_W \in \mathds{R}$. Suppose the (noisy) data is given in a vector space $V$ equipped with an inner product, whose norm we denote by $\|.\|_V$. The noise distribution is assumed to be Gaussian, that is, the probability density function has the form $p(Q) \sim \exp(-\frac{1}{2\sigma_n^2} \|Q - \bar{Q}\|_V^2)$, where $\bar{Q}$ is the true state of the system and $\sigma_n \in \mathds{R}$. Suppose experiments at times $t_i$, $i = 1,\ldots,l$ measuring the trajectory $g(t) Q_0$ have given results $Q_{t_i}$. Then the minimization of
\begin{align}
  S = \int_0^{t_l} \ell(\xi, \dot{\xi}) \, {\rm d}t + \frac{\sigma_W^2}{2\sigma_n^2}\sum_{i = 1}^l \left\|g(t_i)Q_0 - Q_{t_i}\right\|^2_V, \notag
\end{align}
with $\ell$ as in \eqref{Red_Lag_cubics}, can formally be understood as the maximization of the (logarithm of the) probability of the path $g(t)$, given the measurements. Alternative models of stochastic forcing will typically lead to minimization problems of the same type, but with different choices of $\ell$ (see Remark \ref{RB_stoch} below).
\end{remark} 
\subsection{Rigid body splines}\label{Sec-RB_splines}
Let the Lie group $G$ be the set of rigid rotations $SO(3)$, and let $Q$ be the unit sphere $S^2 \in \mathds{R}^3$. We use vector notation for the Lie algebra $\mathfrak{so}(3)$ of the Lie group of rotations $SO(3)$, as well as for its dual $\mathfrak{so}(3)^*$. One identifies $\mathfrak{so}(3)$ with $\mathds{R}^3$ via the familiar isomorphism
\begin{equation}\label{hat_map}
\,\widehat{\,}:  \mathds{R}^3\rightarrow \mathfrak{so}(3),\quad\mathbf{\Omega}=\left(
\begin{array}{c}
a\\
b\\
c\\
\end{array}\right) \mapsto \Omega:=\widehat{\mathbf{\Omega}}=\left(
\begin{array}{ccc}
0&-a&b\\
a&0&-c\\
-b&c&0
\end{array}
\right),
\end{equation}
called the \emph{hat map}.
This is a Lie algebra isomorphism when the vector cross product $\times$ is used as the Lie bracket operation on $\mathds{R}^3$. The identification of $\mathfrak{so}(3)$ with $\mathds{R}^3$ induces an isomorphism of the dual spaces $\mathfrak{so}(3)^* \cong \left(\mathds{R}^3\right)^* \cong \mathds{R}^3$, with the dot product as duality pairing. Let $\gamma$ be a left-invariant Riemannian metric on $SO(3)$. This defines an inner product on $\mathfrak{so}(3)$ which can be expressed as
\begin{align}
  \gamma_e(\mathbf{\Omega}_1, \mathbf{\Omega}_2) = \mathbf{\Omega}_1 \cdot \mathds{I}\mathbf{\Omega}_2 \notag
\end{align}
for a symmetric, positive definite matrix $\mathds{I}$.
The geodesic equation $D_t\dot{g} = 0$ in Euler--Poincar\'e form is
\begin{align}
  \dot{\mathbf{\Omega}} + \mathds{I}^{-1}(\mathbf{\Omega} \times \mathds{I}\mathbf{\Omega}) = 0, \qquad \dot{g} =g \widehat{\mathbf{\Omega}} . \label{EP_RB}
\end{align}
Consequently, the Lagrangian \eqref{Red_Lag_cubics} takes the form
\begin{align}
  \ell(\mathbf{\Omega}^0, \mathbf{\Omega}^1) &= \frac{1}{2} (\mathbf{\Omega}^1 + \mathds{I}^{-1}(\mathbf{\Omega}^0 \times \mathds{I}\mathbf{\Omega}^0)) \cdot \mathds{I} (\mathbf{\Omega}^1 + \mathds{I}^{-1}(\mathbf{\Omega}^0 \times \mathds{I}\mathbf{\Omega}^0)) \notag \\
&= \frac{1}{2} \left\|\mathbf{\Omega}^1 + \mathds{I}^{-1}(\mathbf{\Omega}^0 \times \mathds{I}\mathbf{\Omega}^0) \right\|^2_{\mathfrak{so}(3)}. \label{RB_Lag}
\end{align}
Consider the inexact trajectory planning problem in the left-action, left-reduction form. That is, suppose an initial point $\mathbf{x}_{0}$ and targets $\mathbf{x}_{t_1}, \ldots \mathbf{x}_{t_l} \in S^2$ are given, as well as a tolerance parameter $\sigma$. We seek the minimizer of
\begin{align}
  S[g, \mathbf{\Omega}^0, \mathbf{\Omega}^1, \boldsymbol{\mu}^0, \boldsymbol{\mu}^1] &= \int_{t_0}^{t_l} \frac{1}{2} \left\|\mathbf{\Omega}^1 + \mathds{I}^{-1}(\mathbf{\Omega}^0 \times \mathds{I}\mathbf{\Omega}^0)\right\|_{\mathfrak{so}(3)}^2 + \left<\boldsymbol{\mu}^0, g^{-1}\dot{g} - \mathbf{\Omega}^0\right> \notag \\  &\quad + \left<\boldsymbol{\mu}^1, \dot{\mathbf{\Omega}}^0 - \mathbf{\Omega}^1\right> + \frac{1}{2\sigma^2} \sum_{i = 1}^l \left\|g(t_i) \mathbf{x}_0 - \mathbf{x}_{t_i}\right\|^2. \label{RB_trajplan}
\end{align}
The physical interpretation is as follows. The group of rigid rotations, $SO(3)$, is the configuration manifold of a rigid body constrained to rotate around a fixed point.
In the absence of external torques the motion is governed by the geodesic equation \eqref{EP_RB}, where $\mathds{I}$ is the moment of inertia tensor (see, for example, \cite{Holm_Book2_2011}, Section 2.4). The resulting curve $g(t)$ describes the orientation of the rigid body relative to a space-fixed reference frame. 

Suppose the motion of a rigid body (with or without external torque) is partially observed in an experiment. At discrete times $t_i$ the direction of a particular body fixed axis is measured in the space-fixed frame, generating a sequence of outcomes $\mathbf{x}_{t_i} \in S^2$. Therefore, if $\mathbf{x}_0$ is the initial direction of the axis and $g(t)$ describes the rigid body motion, then $g(t_i)\mathbf{x}_0 - \mathbf{x}_{t_i} = 0$, up to measurement error. One would like to model the trajectory $g(t)$, taking into account this information. The action functional \eqref{RB_trajplan} encodes one such model, yielding the curve $g(t)$ of minimal external torque (in the $L^2$ sense) that is consistent with the experiment. A natural choice for the parameter $\sigma^2$ is to set it equal to the variance of the measurement. Example simulations are shown in  Figures \ref{sphere_fig} and  \ref{solid_fig}. These were generated by the numerical algorithm discussed in Section \ref{Sec_Geom_Dis}. 
\begin{SCfigure}
\centering
\includegraphics[scale=0.229]{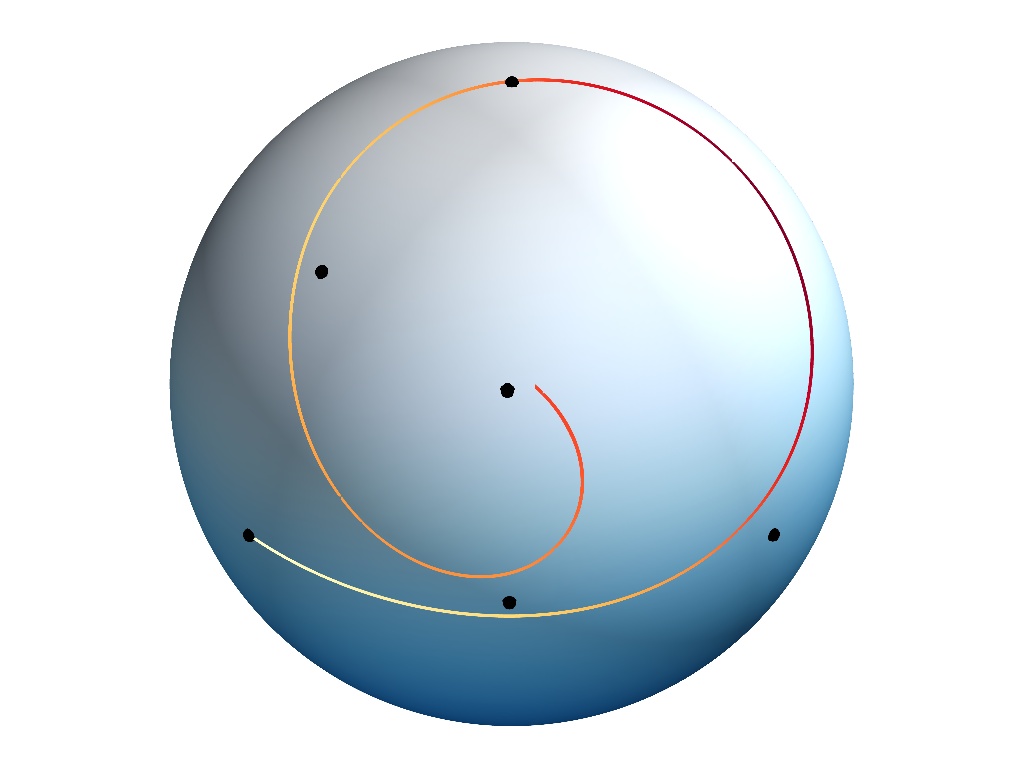}
\caption{\footnotesize Trajectory planning on the sphere through rigid rotations.  The moment of inertia tensor was taken to be the identity matrix and the tolerance parameter was set to $\sigma = 0.025$. The given data points are represented as black dots with uniform time separation. The curve shown is $g(t)\mathbf{x}_0$ for the optimal curve $g(t)$, generated using the algorithm developed in Section \ref{Sec_Geom_Dis}. The coloring represents the local speed along the curve in $SO(3)$, that is, $\|\mathbf{\Omega}^0\|_{\mathfrak{so}(3)}$ (red is large, white is small).}
\label{sphere_fig}
\end{SCfigure}
\begin{figure}
\centering
\includegraphics[scale=0.3]{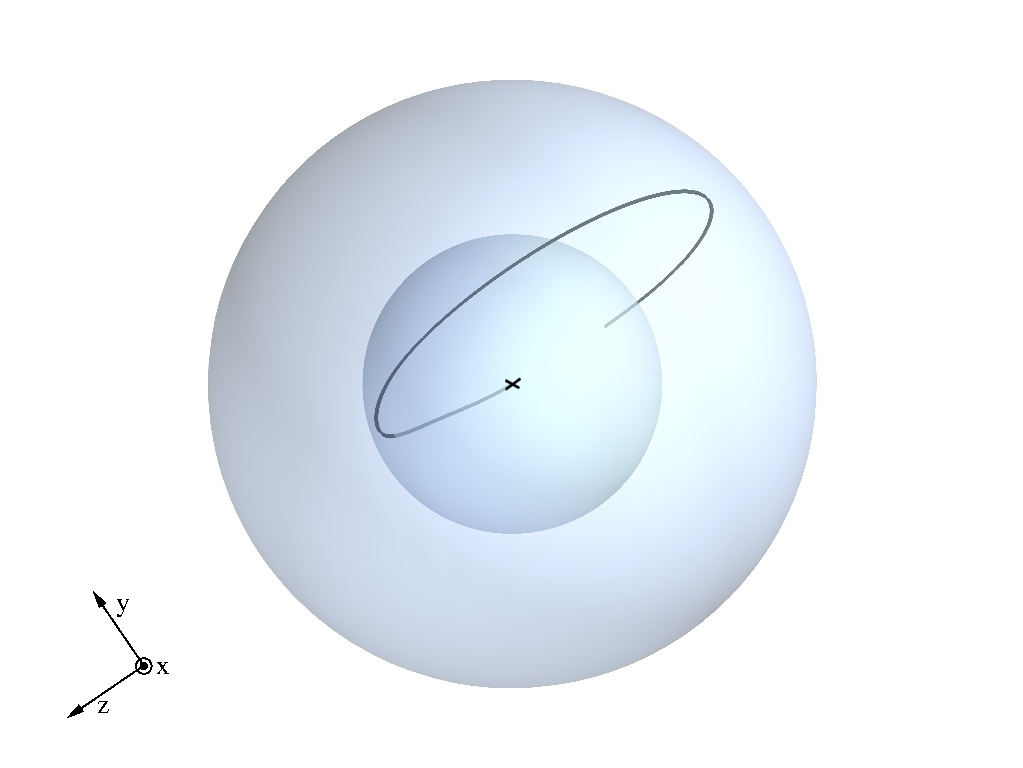}
\caption{\footnotesize Optimal curve $g(t)$ in the group $SO(3)$ of rigid rotations corresponding to the data points shown in Figure \ref{sphere_fig} and generated using the algorithm discussed in Section \ref{Sec_Geom_Dis}. The moment of inertia tensor was taken to be the identity matrix and the tolerance parameter was set to $\sigma = 0.025$. In this figure a given element of $SO(3)$ corresponds to a point on the radial line along the rotation axis, at a distance from the origin equal to the rotation angle. The radius of the inner sphere is $\pi$ and the radius of the outer sphere is $2\pi$. The center, marked by a cross, and the boundary of the outer sphere thus  both represent the identity matrix.}
 \label{solid_fig}
\end{figure}

If the observer is instead moving with a body fixed frame and measuring a space fixed direction, then $g(t_i)^{-1}\mathbf{x}_0 - \mathbf{x}_{t_i} = 0$, up to measurement error. In this case the inexact trajectory planning problem presents itself in the right-action, left-reduction form. Evidently, the formalism presented in this paper applies to any (sufficiently smooth) choice of Lagrangian. For example,
\begin{align}
  \ell(\mathbf{\Omega}^0, \mathbf{\Omega}^1) = \frac{1}{2} \mathbf{\Omega}^0 \cdot ( \mathds{I}\mathbf{\Omega}^1 + \mathbf{\Omega}^0 \times \mathds{I}\mathbf{\Omega}^0) \notag
\end{align}
leads to optimal curves $g(t)$ with minimal work done by external torques.
\begin{remark}\label{RB_stoch}
We mentioned in Remark \ref{TrViRemark} that the minimization of \eqref{RB_trajplan} is related to a certain inverse problem given the stochastic evolution \eqref{SDE}. Alternative stochastic models lead to forms of $\ell$ different from \eqref{RB_Lag}. For example,
\begin{align}
   {\rm d} {\mathbf{\Omega}}  = - \mathds{I}^{-1}(\mathbf{\Omega} \times \mathds{I}\mathbf{\Omega}) {\rm d}t + \sigma_W {\rm d}\mathbf{W}, \notag
\end{align}
where ${\rm d}\mathbf{W} = ({\rm d} W^1, {\rm d} W^2, {\rm d} W^3)^T$ is a vector of independent Brownian motions,
leads to
\begin{align}
  \ell(\mathbf{\Omega}^0, \mathbf{\Omega}^1) =  \frac{1}{2} \left\|\mathbf{\Omega}^1 + \mathds{I}^{-1}(\mathbf{\Omega}^0 \times \mathds{I}\mathbf{\Omega}^0) \right\|^2 \notag
\end{align}
with $\|.\|$ being the  Euclidean norm.
\end{remark}
\subsection{Quantum splines}\label{Sec_QS}
In this section we give an overview of a quantum mechanical application presented in \cite{BrHoMe2012}, where more details can be found. We consider an $n+1$ level quantum system with Hilbert space ${\mathcal H}={\mathds C}^{n+1}$. Quantum state space is the $n$-dimensional complex projective space $\mathds{CP}^n$, which is defined as ${\mathds C}^{n+1} - \left\{0\right\}$ modulo the equivalence relation $|\psi\rangle \sim \lambda |\psi\rangle$ for any complex number $\lambda \neq 0$. The standard geodesic distance is given by
\begin{align}
  d(\psi, \phi) = 2 \arccos \sqrt{\frac{\langle \psi | \phi \rangle \langle \phi | \psi \rangle
}{\langle \psi|\psi\rangle \langle \phi | \phi \rangle}}. \notag
\end{align}
Here we used Dirac notation to write $|\psi\rangle$ for a vector in $\mathds{C}^{n+1}$ and $\langle \psi |$ for its Hermitian conjugate. The evolution of a quantum state is given by the \emph{Schr\"odinger equation}
\begin{eqnarray}
\frac{\rd}{\rd t}|\psi\rangle = -\ri H |\psi\rangle, \notag
\end{eqnarray}
 where the Hamiltonian operator $H$ is a Hermitian matrix. Equivalently one may express this as a differential equation on the unitary group $U(n+1)$ by
\begin{eqnarray}
\dot{U} = -\ri H U, \quad 
|\psi(t)\rangle = U(t) |\psi(0)\rangle. \label{Schrodinger_2}
\end{eqnarray}
One may assume without loss of generality that the anti-Hermitian matrix $\ri H$ is of zero trace, therefore $U(t)$ is in the special unitary group $SU(n+1)$. This is due to the fact that the trace contributes a complex phase factor to the state evolution, which can be neglected in projective terms. Notice that $-\ri H$ is skew-Hermitian and trace free and therefore lies in the Lie algebra $\mathfrak{su}(n+1)$. 

Suppose a time dependent Hamiltonian $H(t)$ is controlled in an experiment whose goal it is to steer a system from initial state $|\psi_0\rangle$ through states $|\psi_{t_j}\rangle$ at times $t_j$, for $j = 1, \ldots, l$. One would like to achieve this trajectory with least change to the experimental apparatus. We formalize this requirement using an inner product on $\mathfrak{su}(n+1)$,
\begin{align}
  \gamma_e(A, B) = - 2 \tr(A B).
\end{align}
 The cost we associate with a time dependent Hamiltonian $H(t)$ is $\int \frac{1}{2} \gamma_e(\dot{H}, \dot{H}) \, {\rm d}t$. We note that $\gamma_e$ extends to a bi-invariant Riemannian metric $\gamma$ on $SU(n+1)$, which means in particular that $\operatorname{ad}^\dagger = -\operatorname{ad}$. The Lagrangian $\ell(H^0, H^1) = \frac{1}{2} \gamma_e(H^1, H^1)$ is therefore equal to the reduced Lagrangian \eqref{Red_Lag_cubics} for the metric $\gamma$ on $SU(n+1)$. The total cost functional takes the left-action, right-reduction form
\begin{align}
  S[U, \ri H^0, \ri H^1, M^0, M^1] &= \int_{0}^{t_l} \tr(H^1 H^1) + \left<M^0, \dot{U}U^{-1} - \ri H^0\right> + \left<M^1, \ri \dot{H}^0 - \ri H^1\right> \, {\rm d}t \notag \\
&\quad + \frac{1}{2\sigma^2} \sum_{i = 1}^l d^2(U(t_i) \psi_0, \psi_{t_i}). \notag
\end{align}
The tolerance parameter $\sigma$ can be used to trade off quality of matching against a reduced cost associated with change in the Hamiltonian over time. For more details, see \cite{BrHoMe2012}.
\subsection{Macromolecular configurations}\label{Macro}
Equilibrium configurations of macromolecular structures and of DNA in particular can be modelled using the classical theory of elastic rods. See \cite{DNASupercoil, KiChi2006} for examples of this approach. In this section we formulate an inexact trajectory planning problem in this context.

We start by describing how the configuration of an elastic rod can be described by a position curve $\mathbf{r}(s)$ in $ \mathds{R}^3$ and a curve $R(s)$ in the group of rigid rotations, $ SO(3)$. Here, $s \in [0, 1]$ parameterizes the cross sections of the rod along its length, whereby for a given value of $s$ the vector $\mathbf{r}(s)$ points to the center of mass of the respective cross section, as seen in the lab frame.
Let $\mathbf{e}_i(s)$, $i = 1,\, 2,\, 3$ be an orthonormal frame such that $\mathbf{e}_1(s)$ and $\mathbf{e}_2(s)$  point along the principal axes of the moment of inertia tensor of the cross section. We will refer to this frame as the \emph{body-fixed} frame. $R(s)$ is the rotation that transforms the initial frame at $s = 0$  (which we assume to coincide with the lab frame) to the body-fixed frame at $s$. Therefore the configuration of a macromolecule can be described by a curve $g(s) = (R(s), \mathbf{r}(s))$ in the special Euclidean group, $SE(3)$, originating at the identity. The group multiplication rule of $SE(3)$ is
\begin{align}
  (R_1, \mathbf{r}_1) (R_2, \mathbf{r}_2) = (R_1R_2, R_1\mathbf{r}_2 + \mathbf{r}_1) \notag.
\end{align}
The Lie algebra $\mathfrak{se}(3)$ consists of elements $(\Omega, \mathbf{v})$ with $\Omega \in \mathfrak{so}(3)$ and $\mathbf{v} \in \mathds{R}^3$. Applying the inverse of the hat map \eqref{hat_map} to $\Omega$ we can represent $\mathfrak{se}(3)$ as $\mathds{R}^6$. The $\operatorname{ad}$ operation becomes
\begin{align}
  \operatorname{ad}_{\boldsymbol{\xi}_1} \boldsymbol{\xi}_2 = \operatorname{ad}_{(\mathbf{\Omega}_1, \mathbf{v}_1)} (\mathbf{\Omega}_2, \mathbf{v}_2) = (\mathbf{\Omega}_1 \times \mathbf{\Omega}_2,\, \mathbf{\Omega}_1 \times \mathbf{v}_2 - \mathbf{\Omega}_2 \times \mathbf{v}_1). \notag
\end{align}
If we identify the dual $\mathfrak{se}(3)^*$ with $\mathds{R}^6$ using the standard dot product as duality pairing, then
\begin{align}
  \operatorname{ad}^*_{(\boldsymbol{\Omega}, \mathbf{v})} (\boldsymbol{\mu}, \mathbf{a}) = (-\mathbf{\Omega} \times \boldsymbol{\mu} - \mathbf{v} \times \mathbf{a}, - \mathbf{\Omega} \times \mathbf{a}). \label{ad_star}
\end{align}
 The \emph{body-fixed velocity} vector pertaining to a configuration $(R(s), \mathbf{r}(s))$ is defined as
\begin{align}
  \xi = g^{-1} \dot{g} = (R^{-1} \dot{R}, R^{-1} \dot{\mathbf{r}}) = (\Omega, R^{-1}\dot{\mathbf{r}}) \notag
\end{align}
where a superscript $\cdot$ means derivation with respect to $s$. In vector notation we can set $\mathbf{v} = R^{-1}\dot{\mathbf{r}}$ so that $\boldsymbol{\xi} = (\mathbf{\Omega}, \mathbf{v}) \in \mathds{R}^6$. 

A macromolecule that is experimentally constrained to assume a configuration with given final rotation and displacement $g(1) = (R(1), \mathbf{r}(1))$ will relax into an equilibrium state that minimizes potential energy with respect to all possible configurations respecting the constraint. For the case of DNA the authors of \cite{KiChi2006} propose to model this effect using the Lagrangian $L: TSE(3) \to \mathds{R}$ whose left-reduced form $l: \mathfrak{se}(3) \to \mathds{R}$ is given by
\begin{align}
  l(\boldsymbol{\xi}) = \frac{1}{2} (\boldsymbol{\xi} - \mathbf{z}) \cdot K (\boldsymbol{\xi} - \mathbf{z}). \notag
\end{align}
The $6\times 6$ matrix $K$ is symmetric and positive definite and encodes the various stiffness properties. The double helix structure of DNA means that the equilibrium configuration for unconstrained end points retains a number $n$ of rotations along its length. This is expressed by the vector
\begin{align}
  \mathbf{z} = \left(\begin{array}{c} 2 \pi n \mathbf{e}_z \\ \mathbf{e}_z\end{array}\right),
\end{align}
where we assume without loss of generality that the equilibrium configuration for unconstrained end points is oriented along the spatial $z$-axis and has unit length. 

In the case of constrained final rotation and displacement $g(1) =  (R(1), \mathbf{r}(1))$ the equilibrium configuration minimizes the action functional $S = \int_0^1 l(\boldsymbol{\xi})\, {\mathrm d}s$. Hamilton's principle $\delta S = 0$ leads to Euler--Poincar\'e equations
\begin{align}
  \dot{\boldsymbol{\xi}} = K^{-1}\operatorname{ad}^*_{\boldsymbol{\xi}}K (\boldsymbol{\xi} - \mathbf{z}) = \operatorname{ad}^\dagger_{\boldsymbol{\xi}}(\boldsymbol{\xi} - \mathbf{z}), \label{Strand_EP}
\end{align}
with $\operatorname{ad}^*$ operation given by \eqref{ad_star} and the operation $\operatorname{ad}^\dagger$ defined by $\operatorname{ad}^\dagger_{\boldsymbol{\xi}_1}\boldsymbol{\xi}_2 := K^{-1} \operatorname{ad}^*_{\boldsymbol{\xi}_1} K \boldsymbol{\xi}_2$. This equation of motion can be used to design second order Lagrangians to model non-equilibrium states of the DNA. For example, let us set
\begin{align}
  \ell(\boldsymbol{\xi}^0, \boldsymbol{\xi}^1) &= \frac{1}{2} (\boldsymbol{\xi}^1 - \operatorname{ad}^\dagger_{\boldsymbol{\xi}^0}(\boldsymbol{\xi}^0 - \mathbf{z})) \cdot K (\boldsymbol{\xi}^1 - \operatorname{ad}^\dagger_{\boldsymbol{\xi}^0}(\boldsymbol{\xi}^0 - \mathbf{z}))\notag \\ &= \frac{1}{2} \left\|\boldsymbol{\xi}^1 - \operatorname{ad}^\dagger_{\boldsymbol{\xi}^0}(\boldsymbol{\xi}^0 - \mathbf{z})\right\|^2_{K}. \notag
\end{align}
Suppose an experiment measures the position of the center of mass $\mathbf{r}(s_i)$ at a number of parameter values $s_i$ ($i = 1, \ldots, l$) as well as a body fixed direction, say $\mathbf{e}_3 (s_i)$. The space of measurement outcomes is $Q = S^2 \times \mathds{R}^3 \in \mathds{R}^6$ with $SE(3)$ action given by
\begin{align}
  (R, \mathbf{r}) (\widehat{\mathbf{x}}, \mathbf{y}) = (R \widehat{\mathbf{x}}, R\mathbf{y} + \mathbf{r}). \notag
\end{align}
If the measurements yield the sequence $(\widehat{\mathbf{x}}_{s_i}, \mathbf{y}_{s_i})$ ($ i = 1, \ldots, l)$ this suggests that up to measurement error the configuration $g(s) \in SE(3)$ satisfies
  \begin{align}
    g(s_i)(\mathbf{e}_3(0), \mathbf{0}) = (\widehat{\mathbf{x}}_{s_i}, \mathbf{y}_{s_i}). \notag
  \end{align}
The task of modelling the configuration $g(s)$ can then be cast in the form of an inexact trajectory planning problem in the left-action, left-reduction form with cost functional
\begin{align}
  S[g, \boldsymbol{\xi}^0, \boldsymbol{\xi}^1, \boldsymbol{\mu}^0, \boldsymbol{\mu}^1] &= \int_0^{t_l}\frac{1}{2} \left\|\boldsymbol{\xi}^1 - \operatorname{ad}^\dagger_{\boldsymbol{\xi}^0}(\boldsymbol{\xi}^0 - \mathbf{z})\right\|^2_{K} + \left<\boldsymbol{\mu}^0, g^{-1}\dot{g} - \boldsymbol{\xi}^0\right> + \left<\boldsymbol{\mu}^1, \dot{\boldsymbol{\xi}}^0 - \boldsymbol{\xi}^1\right> \, {\rm d}s \notag \\
&\quad + \frac{1}{2 \sigma^2} \sum_{i= 1}^l \left\|  g(s_i)(\mathbf{e}_3(0), \mathbf{0}) - (\widehat{\mathbf{x}}_{s_i}, \mathbf{y}_{s_i})\right\|^2. \notag
\end{align}
\begin{remark}
Due to the anisotropy in velocity space, expressed by the vector $\mathbf{z}$, the Euler--Poincar\'e equation \eqref{Strand_EP} is \emph{not} the reduced geodesic equation for the curve $g(t)$ with respect to the metric defined by $K$. Consequently, the solution to the inexact trajectory matching problem is not a Riemannian cubic spline.
\end{remark}
\section{Geometric discretization}\label{Sec_Geom_Dis}
The purpose of this section is to illustrate how the higher-order Hamilton--Pontryagin principle offers a direct route towards geometric numerical integrators.
  All one needs to do, in essence, is to provide a geometric discretization of the constraint $TR_{g^{-1}}\dot{g} - \xi^0 = 0$ and define a discrete Hamilton--Pontryagin principle accordingly. For first order variational problems this idea was introduced in \cite{BRMa09}, building on the general theory of variational integrators (see \cite{MaWe01} for an extensive review). We follow in the footsteps of  \cite{BRMa09} to treat second order problems. Third and higher orders can be dealt with in a similar fashion.

Our main motivation for the development of a geometric integrator lies with the exact momentum behaviour exhibited by the discrete solution curves. Not only does the momentum conservation on open intervals \eqref{Conservation} translate exactly into the discrete picture, but the behavior at the nodes is given by discrete versions of the continuous time node equations \eqref{node_1}--\eqref{node_4}. As a consequence, one can obtain discrete analogues of Theorem \ref{Theorem_m0} and Corollary \ref{Corollary}. This means that the numerical search for the optimal initial value of the momentum $\mu^0$ can be restricted to a linear subspace of $\mathfrak{g}^*$ of the same dimension as the data manifold $Q$. As we shall see below, the variational nature of the integrator also means that the discrete flow map preserves the canonical symplectic form on $T^*(TG)$.
\subsection{A geometric integrator}
In discrete mechanics the time axis $[t_0, t_l]$ is replaced by a set of discrete time points $t_k = t_0 + kh$, $k = 0, \ldots , N$, where $h$ is the step size and $t_l = t_0 + Nh$. We use integers $N_i$, $i = 1, \ldots, l$, as node indices $t_i = t_0 + N_ih$. For convenience we also define $N_0 := 0$. We will need a map $\tau: \mathfrak{g} \to G$ that approximates the Lie exponential and is an analytic diffeomorphism in a neighbourhood of $0$ with $\tau(0) = e$ as well as $\tau(\xi) \tau(-\xi) = e$ for all $\eta \in \mathfrak{g}$. An example is the \emph{Cayley transform}, which is a second-order approximation of the Lie exponential in quadratic matrix Lie groups. This includes the applications discussed in Section \ref{Sec-Applications}. The Cayley transform is defined as
\begin{align}
  \tau(\xi) = (e - \xi/2)^{-1} (e + \xi/2). \notag
\end{align}
This and further examples are discussed in \cite{BRMa09}. 

Since $\tau$ is an approximate of the Lie exponential, a simple way of discretizing the constraint $TR_{g^{-1}}\dot{g} - \xi^0 = 0$ is to require that $g_{k+1} = \tau(h\xi_k^0) g_k$, where $h$ is the size of a time step. Similarly one may translate $\dot{\xi}^0 = \xi^1$ to $\xi^0_{k+1} = \xi^0_k + h \xi_k^1$. With these considerations in mind we define a discretized version of the action functional \eqref{Cost_LM} on discrete path space $\mathcal{C}_d$. Let
\begin{align}
  \mathcal{C}_d = \left\{ g_0, \xi_0^0, \xi^1_0, (g_k, \xi^0_k, \xi^1_k, \check{\mu}^0_k, \mu^1_k)_{k=1}^{N}  \right\} = (G \times 2\mathfrak{g}) \times (G \times 2\mathfrak{g} \times 2 \mathfrak{g}^*)^N, \notag
\end{align}
then we define $S_d: \mathcal{C}_d \to \mathds{R}$ as
\begin{align}
  S_d &= h \left[\sum_{k=0}^{N-1} \ell(\xi^0_k, \xi^1_k) + \left< \check{\mu}_{k+1}^0, \frac{1}{h} \tau^{-1}(g_{k+1}g_k^{-1}) - \xi^0_k\right> + \left<\mu^1_{k+1}, \frac{1}{h}(\xi^0_{k+1} - \xi^0_k) - \xi^1_k\right>\right] \notag \\
&\quad + \frac{1}{2\sigma^2} \sum_{i=1}^l d^2(g_{N_i}Q_0, Q_{t_i}). \label{mismatch_penalty}
\end{align}
In analogy to continuous time we assume that  $g_0 = e$ and $\xi_0^0$ are given. 

The discrete Euler--Lagrange equations follow from Hamilton's principle $\delta S_d = 0$. That is, they characterize paths $\gamma \in \mathcal{C}_d$, for which $ \delta S_d := \delta \gamma(S) = 0$ for all variations $\delta \gamma \in T_{\gamma}\mathcal{C}_d$ with $\delta g_0 = 0$ and $\delta \xi_0^0 = 0$. In the process of computing $\delta S_d$ we need to calculate $\delta \tau^{-1}(g_{k+1}g_k^{-1})$. For that purpose it is convenient to introduce the left-trivialized differential of $\tau$ at $\xi \in \mathfrak{g}$,
     \begin{equation}
       D\tau_\xi: \mathfrak{g} \to \mathfrak{g}, \quad \eta \mapsto \tau(\xi)^{-1} \left(T_{\xi}\tau(\eta)\right), \notag
     \end{equation}
whose inverse we denote by $D\tau_\xi^{-1}$. By taking a derivative of $\tau(\xi)\tau(-\xi) = e$ one can show that \cite{BRMa09}
\begin{align}\label{Dtau_identity}
   D\tau_{\xi}^{-1} = D\tau_{-\xi}^{-1}\circ \operatorname{Ad}_{\tau(\xi)}. 
\end{align}
Denoting $\eta_k := (\delta g_{k}) g_k^{-1}$ we find
\begin{align}
   \delta \tau^{-1}(g_{k+1}g_k^{-1}) &= T_{\tau(h\xi^0_k)}\tau^{-1}(\eta_{k+1}\tau(h\xi^0_k)) - T_{\tau(h\xi^0_k)}\tau^{-1}(\tau(h\xi^0_k) \eta_k)\notag \\
&= T_{\tau(h\xi^0_k)}\tau^{-1}(\tau(h\xi^0_k) \operatorname{Ad}_{\tau(-h\xi^0_k)}\eta_{k+1}) - T_{\tau(h\xi^0_k)}\tau^{-1}(\tau(h\xi^0_k) \eta_k) \notag \\
&= D\tau_{h\xi^0_k}^{-1} ( \operatorname{Ad}_{\tau(-h\xi^0_k)}\eta_{k+1)}) -  D\tau_{h\xi^0_k}^{-1}( \eta_k)\notag \\
&=  D\tau_{-h\xi^0_k}^{-1} (\eta_{k+1}) -   D\tau_{h\xi^0_k}^{-1}( \eta_k), \notag
\end{align}
where in the last equality we used \eqref{Dtau_identity}. 
Introducing the quantities
\begin{align} \label{Defs_0}
  \mu^1_0 := \mu^1_1 + h\check{\mu}^0_1 - h \frac{\delta \ell}{\delta \xi^0_0}\,, \qquad \mu_0^0 := (D\tau^{-1}_{h\xi_0^0})^* \check{\mu}^0_1
\end{align}
and
\begin{align}
  \mu^0_k := (D\tau_{-h\xi_{k-1}^0}^{-1})^*\check{\mu}_k^0 \qquad (k = 1, \ldots, N), 
\end{align}
we obtain, after rearranging terms,
\begin{align}
  \delta S_d &= h\left[\sum_{k = 1}^{N-1} \left< \frac{\delta \ell}{\delta \xi_k^0} - \check{\mu}^0_{k+1} + \frac{1}{h} \mu^1_k - \frac{1}{h} \mu^1_{k+1}, \delta \xi^0_k\right> + \left< \frac{\delta \ell}{\delta \xi_k^1} - \mu^1_{k+1}, \delta \xi^1_k\right> \right.  \notag\\
 &\quad \left.+ \left<\frac{1}{h} \mu^0_k - \frac{1}{h} (D\tau^{-1}_{h\xi^0_k})^* (D\tau_{-h\xi^0_k})^*\mu^0_{k+1}\right> + \left<\delta \check{\mu}^0_{k+1}, \ldots\right> + \left<\delta \mu^1_{k+1}, \ldots \right>\right] \notag \\
&\quad + h\left<\frac{\delta \ell}{\delta \xi^1_0} - \mu^1_1, \delta \xi^1_0\right> + h \left<\delta \check{\mu}_1, \ldots\right> + h\left<\delta \mu_1^1, \ldots \right> + \left<\mu^1_N, \delta \xi^0_N\right>  - \left<\mu^1_0, \delta \xi_0^0\right> \notag \\
&\quad + \left< \mu^0_N, \eta_N\right> - \left<\mu^0_0, \eta_0\right> 
+ \frac{1}{\sigma^2} \sum_{i=1}^l \left<d_{N_i} J^Q({\rm{d}}_1d_{N_i}), \eta_{N_i}\right>\,, \notag
 \end{align}
where we used abbreviations $d_{N_i} = d(g_{N_i}Q_0, Q_{t_i})$ and ${\rm{d}}_1d_{N_i} := {\rm{d}}_1d(g_{N_i}Q_0, Q_{t_i})$. The Euler--Lagrange equations are therefore composed of the following equations. The constraints
\begin{align} \label{constraints}
  g_{k+1} = \tau(h\xi^0_k) g_k, \qquad \xi^0_{k+1} = \xi^0_k + h\xi_k^1 \qquad (k = 0, \ldots, N-1),
\end{align}
the discrete equations for the Legendre--Ostrogradsky momenta
\begin{align} 
  &\mu^1_{k+1} = \mu^1_k -h (D\tau_{-h\xi^0_{k}})^*\mu^0_{k+1} + h\frac{\delta \ell}{\delta \xi^0_k} \qquad (k = 1, \ldots, N-1) \label{Ostro} \\
&\frac{\delta \ell}{\delta \xi_k^1} - \mu^1_{k+1} = 0 \qquad (k = 0, \ldots, N-1), \label{LegTr_highest}
\end{align}
the discrete version of the Euler--Poincar\'e equation for interior indices $k \neq N_i$ $(i = 1, \ldots, l)$
\begin{align}
  \mu^0_{k+1} = (D\tau^{-1}_{-h\xi^0_k})^* (D\tau_{h\xi^0_k})^* \mu^0_k \label{EP_interior}
\end{align}
and for node indices $k = N_i$ $(i = 1, \ldots, l-1)$
\begin{align}
  \mu^0_{k+1} = (D\tau^{-1}_{-h\xi^0_k})^* (D\tau_{h\xi^0_k})^* \left(\mu^0_k + \frac{d_k}{\sigma^2} J^Q({\rm{d}}_1d_{k})\right). \label{EP_node}
\end{align}
Finally,
\begin{align}
  &\mu^0_N +  \frac{d_N}{\sigma^2} J^Q({\rm{d}}_1d_{N}) = 0, \label{End_mu0}\\
  &\mu^1_N = 0. \label{End_mu1}
\end{align}
A solution $\gamma \in \mathcal{C}_d$ to \eqref{constraints}--\eqref{EP_node} is said to have \emph{initial conditions} $(g_0, \xi^0_0, \mu^0_0, \mu^1_0) \in G \times \mathfrak{g} \times 2\mathfrak{g}^* \cong T^*(TG)$, using the definitions \eqref{Defs_0}. If the Lagrangian $\ell$ is hyperregular, then \eqref{LegTr_highest} can be solved for $\xi^1_k$. This means that $\xi^1_k$ can be eliminated from equations \eqref{constraints}--\eqref{EP_node}, which can subsequently be integrated for given initial conditions. If in addition equations \eqref{End_mu0} and \eqref{End_mu1} are satisfied, then $\gamma$ is a critical point of the action functional $S_d$.
\begin{remark}
In Remark \ref{Rem_1} we mentioned that besides the above left-action, right-reduction case, three other cases were available to be considered. In a similar manner to what we observed in that remark the right-action, right-reduction case introduces changes to equations \eqref{EP_node} and \eqref{End_mu0}. These become
\begin{align}
   &\mu^0_{k+1} = (D\tau^{-1}_{-h\xi^0_k})^* (D\tau_{h\xi^0_k})^* \left(\mu^0_k - \frac{d_k}{\sigma^2}\operatorname{Ad}^*_{g_k^{-1}} J^Q({\rm{d}}_1d_{k})\right) \quad \mbox{and} \notag \\
&\mu^0_N -  \frac{d_N}{\sigma^2} \operatorname{Ad}^*_{g_N^{-1}}J^Q({\rm{d}}_1d_{N}) = 0, \notag
\end{align}
respectively. The remaining two cases (left-action, left-reduction; and right-action, left-reduction) are equivalent to the first two in  just the same way as explained in Remark \ref{Rem_1}.
\end{remark}
\subsection{Geometric properties}
As in Section \ref{HOHP_and_geometry_of_mult} the interior equations are conveniently analyzed by omitting the mismatch penalty term \eqref{mismatch_penalty} from the action functional, so that
\begin{align}
 \mathcal{J}_d &= h \left[\sum_{k=0}^{N-1} \ell(\xi^0_k, \xi^1_k) + \left< \check{\mu}_{k+1}^0, \frac{1}{h} \tau^{-1}(g_{k+1}g_k^{-1}) - \xi^0_k\right> + \left<\mu^1_{k+1}, \frac{1}{h}(\xi^0_{k+1} - \xi^0_k) - \xi^1_k\right>\right] \notag
\end{align}
 The arguments that surround equation \eqref{Symplectic_arg} and show symplecticity of the continuous time flow can then be applied in a straightforward manner to the discrete case. Indeed, interior equations \eqref{constraints}--\eqref{EP_interior} define a flow map $F_d: G \times \mathfrak{g} \times 2\mathfrak{g}^* \to  G \times \mathfrak{g} \times 2\mathfrak{g}^*$, which integrates a solution $\gamma$ for given initial conditions. That is, $(F_d)^k(g_0, \xi^0_0, \mu^0_0, \mu^1_0) = (g_k, \xi^0_k, \mu^0_k, \mu^1_k)$ or more succinctly $(F_d)^k(\gamma_0) = \gamma_k$. We restrict $\mathcal{J}_d$ to solutions of \eqref{constraints}--\eqref{EP_interior} and express it as a function $\mathcal{J}_{d, \mbox{initial}}: T^*(TG) \to \mathds{R}$ of initial conditions $\gamma_0 \in T^*(TG)$. This means that if $\gamma \in \mathcal{C}_d$ is the solution obtained by integrating $\gamma_0$ then ${J}_{d, \mbox{initial}}(\gamma_0) = \mathcal{J}_d(\gamma)$. It follows that
 \begin{align}
  {\rm d}\mathcal{J}_{d, \mbox{initial}} (\delta \gamma_0) = \left<\mu^1_N, \delta \xi^0_N\right>  - \left<\mu^1_0, \delta \xi_0^0\right> + \left< \mu^0_N, \eta_N\right> - \left<\mu^0_0, \eta_0\right> = (((F_d)^k)^* \theta - \theta) (\delta \gamma_0), \notag
 \end{align}
where $\theta$ is the canonical one-form on $T^*(TG)$. Taking an exterior derivative shows that the canonical symplectic form $\omega = {\rm d}\theta$ is preserved by $F_d$. Hence the discrete flow map $F_d$ is symplectic.

Similarly the observations given in Section \ref{Section_MCNT} can be translated to the discrete picture. Indeed we pointed out in the paragraph of equation \eqref{Noether_var} how to obtain the conservation of the momentum $J_k = \operatorname{Ad}^*_g \mu^0$ from a variational perspective. These arguments can be applied to the discrete variational principle to show that $\operatorname{Ad}^*_{g_{k+1}} \mu^0_{k+1} = \operatorname{Ad}^*_{g_k} \mu^0_k$ for interior indices $k \neq N_i$. Alternatively, a manipulation of equation \eqref{EP_interior} using \eqref{Dtau_identity} shows that
\begin{align}
   \mu^0_{k+1} = (D\tau^{-1}_{-h\xi^0_k})^* (D\tau_{h\xi^0_k})^* \mu^0_k = \operatorname{Ad}^*_{\tau(-h\xi_k^0)}\mu^0_k = \operatorname{Ad}^*_{g_k g_{k+1}^{-1}} \mu^0_k = \operatorname{Ad}^*_{g_{k+1}^{-1}} \operatorname{Ad}^*_{g_k} \mu^0_k, \notag
\end{align}
and therefore $\operatorname{Ad}^*_{g_{k+1}} \mu^0_{k+1} = \operatorname{Ad}^*_{g_k} \mu^0_k$.
\begin{remark}
  The symplectic flow equations on open intervals can equivalently be discretized in the purely Lagrangian picture by following \cite{CoJiMdD2012}. One chooses a suitable discrete Lagrangian $L_d: G \times G \times G \to \mathds{R}$ and applies variational calculus to the discrete action sum
  \begin{align}
    S_d = \sum_{k=0}^{N-2}L_d(g_k, g_{k+1}, g_{k+2}). \notag
  \end{align}
Let us define $\xi: G\times G \to \mathfrak{g}$ by $\xi(g_1, g_2) = h^{-1} \tau^{-1}\left(g_2g_1^{-1}\right)$ and set
\begin{align}
  L_d(g_k, g_{k+1}, g_{k+2}) := h \ell(\xi(g_k, g_{k+1}), h^{-1}(\xi(g_{k+1}, g_{k+2})- \xi(g_{k}, g_{k+1}))). \notag
\end{align}
Boundary conditions being equal, the resulting optimal curve $(g_0, \ldots, g_N)$ is the same as the one we obtained from the discrete Hamilton--Pontryagin principle. The Hamilton--Pontryagin principle has an advantage in situations where more sophisticated discretizations of the constraint $TR_{g^{-1}}\dot{g} = \xi^0$ are chosen. For example, in the preliminary study \cite{BuHoMe2011} Runge--Kutta--Munthe--Kaas methods were used to introduce a class of such integrators. Those integrators can still be understood in the purely Lagrangian framework, however the definition of the corresponding function $L_d(g_k, g_{k+1}, g_{k+2})$ is implicit in that evaluating it requires solving a variational problem. The Hamilton--Pontryagin approach circumvents this difficulty by building the discretization of the constraint into the variational principle from the outset.
\end{remark}
The node equation \eqref{EP_node} reflects in a geometrically consistent way the jump discontinuities of $\operatorname{Ad}^*_g\mu^0$ given in \eqref{Mom_jump}. Indeed, \eqref{EP_node} says that
\begin{align}
  J_{k+1} = J_k + \frac{d_k}{\sigma^2} \operatorname{Ad}^*_{g_k} J^Q({\rm{d}}_1d_{k})\notag
\end{align}
when $k = N_i$. Moreover, the final time conditions \eqref{End_mu0} and \eqref{End_mu1} are exact analogues of \eqref{node_3} and \eqref{node_4}.  See Figure \ref{mom_behaviour} for an example.
\begin{figure}[htb]
\begin{center}
\includegraphics[scale=0.7]{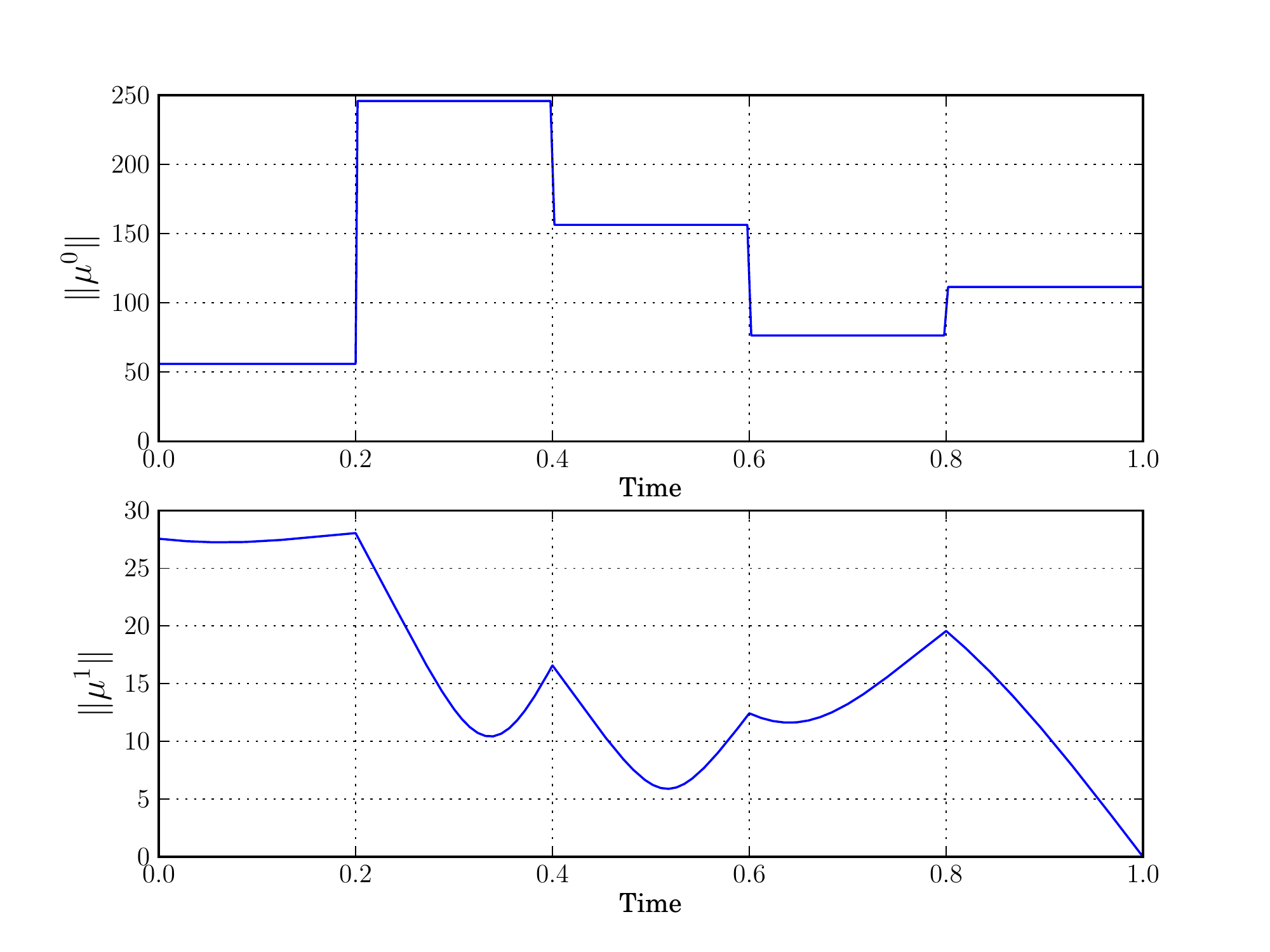}
\end{center}
\caption{\footnotesize Momentum norms. For the interpolating discrete cubic of Figure \ref{solid_fig}, the plot shows the norms of the momenta $\mu^0_k$ and $\mu^1_k$. The norm of $\mu_k$ displays the momentum discontinuities at node indices as well as exact conservation at interior indices, in accordance with \eqref{EP_interior} and \eqref{EP_node}. The norm of $\mu^1_k$ demonstrates continuity, as found in \eqref{Ostro}. Both curves respect terminal conditions \eqref{End_mu0} and \eqref{End_mu1}.}
\label{mom_behaviour}
\end{figure}
 Putting everything together leads to discrete versions of Theorem \ref{Theorem_m0} and Corollary \ref{Corollary}. The proofs are analogous to the continuous time case.
\begin{theorem} \label{Theorem_m0_discrete}
  For $k = 0, \ldots, N$
  \begin{align} \notag
    \mu^0_k = -\frac{1}{\sigma^2} \operatorname{Ad}^*_{g_k^{-1}} \left(\sum_{i = 1}^{l} 1_{k \leq N_i} d_{N_i} \operatorname{Ad}^*_{g_{N_i}} J^Q(\mathrm{d}_1d_{N_i})\right). 
  \end{align}
\end{theorem}
\begin{corollary}\label{Corollary_discrete}
  For $k = 0, \ldots, N$
  \begin{align}
    \left<\mu^0_k, \rho\right> = 0 \qquad \mbox{for all }\, \rho \in \mathfrak{g}_{g_kQ_0}. \notag
  \end{align}
\end{corollary}
\subsection{Practical remarks}\label{Remarks}
The integrator derived above provides a way of finding a numerical solution to the inexact trajectory planning problem. 

The discrete equations of motion \eqref{constraints}--\eqref{EP_node} can be employed to express the action functional $\mathcal{J}_d$ as a function  $\mathcal{J}_{d, \mbox{initial}}: T^*(TG) \to \mathds{R}$ of initial conditions $\gamma_0 = (g_0, \xi_0^0, \mu^0_0, \mu^1_0) \in T^*(TG) \cong G \times \mathfrak{g} \times 2\mathfrak{g}^*$. The minimizer in the space of initial conditions can then be determined by a gradient descent method. Since $g_0 = e$ and $\xi_0^0$ are given, the minimization is in effect only over the variables $(\mu_0^0, \mu_0^1)$. By Corollary \ref{Corollary_discrete} the optimal $\mu_0^0$ lies in the subspace of $\mathfrak{g}^*$ that annihilates $\mathfrak{g}_{Q_0}$, to which the search can therefore be restricted. On the other hand one still needs to consider all of $\mathfrak{g}^*$ for the optimization of $\mu_0^1$.

The gradient of $\mathcal{J}_{d, \mbox{initial}}$ can be estimated via finite-difference methods. However, this requires the repeated forward integration of  \eqref{constraints}--\eqref{EP_node}. The number of such integrations increases with the number of dimensions of the Lie group $G$, and for higher dimensional systems this quickly becomes unfeasible. Such difficulties can be circumvented using the standard method of \emph{adjoint equations}, which can be easily implemented for the geometric discretization presented here (a detailed derivation is provided in the Appendix). Then the \emph{exact} gradient is obtained at the cost of integrating twice (once forward and once backward) a system of equations of the same complexity as the forward equations.

The significance of the exact preservation of final time constraints \eqref{node_3} and \eqref{node_4} in the form of \eqref{End_mu0}, \eqref{End_mu1} is to provide verification that a (local) minimum has been found.
When the tolerance is tight ($\sigma$ is small) and in the absence of a good initial guess it may occur that the algorithm tends to a local minimum rather than the global one. Suitable initial guesses can be computed using a homotopy strategy. This means a step-by-step reduction of $\sigma$, where the optimum at one value of $\sigma$ is taken as the initial guess at the next smaller value. 

We used this algorithm to generate the figures in this paper. For simulations of quantum splines (see Section \ref{Sec_QS}) using the same methods we refer to \cite{BrHoMe2012}.
\section{Discussion and Outlook}
\paragraph{Discussion.} 
This paper has discussed a type of inexact trajectory planning problem whose optimal curves are required to pass near a sequence of fixed target positions at designated times within a certain tolerance. In Section \ref{Sec_Geometry_of_TP} a new derivation of the Euler--Lagrange equations for this type of problem was obtained by applying reduction by symmetry to a higher--order Hamilton--Pontryagin principle. This approach provided a new geometric interpretation of the previously known node equations in terms of Legendre--Ostrogradsky momenta. The highest order momentum was seen to undergo discontinuous jumps at the node times as a consequence of a partially broken Lie group symmetry. This was the content of Theorem \ref{Theorem_m0} and Corollary \ref{Corollary}. In Section \ref{Sec-Applications} several applications of the theory were discussed, which summoned the inexact trajectory planning problem both from a control theoretic viewpoint (quantum splines) as well as in the context of a type of inverse problem (rigid body splines, macromolecular configurations). Finally, Section \ref{Sec_Geom_Dis} was concerned with the numerical approach to solving the problem at hand. The reduced Hamilton--Pontryagin principle was taken as the starting point and a geometric discretization of the Euler--Lagrange equations was obtained, which led to exact momentum behavior and discrete versions of both Theorem \ref{Theorem_m0} and Corollary \ref{Corollary}. This meant in particular that the search for the optimal initial value of the highest order momentum could be restricted to a subspace of the dual of the Lie algebra of the Lie group $G$, whose action describes the motion.

\paragraph{Outlook.} This work invites further development in several directions. First, one can show that the discrete flow map derived in Section \ref{Sec_Geom_Dis} is accurate only to first order in step size $h$. The development of geometric methods with a higher degree of accuracy would be desirable. The class of integrators presented in the preliminary study \cite{BuHoMe2011} may prove useful in this regard. An alternative possibility is the Lagrangian approach of \cite{CoJiMdD2012} together with suitably exact approximations of the Lagrangian function. Second, the final step in the numerical optimization of the cost functional was based on a shooting method with gradient descent in the space of initial conditions. A comparison with more sophisticated methods of nonlinear programming would be a useful guide for further development. Third, the example of the molecular strand (Section \ref{Macro}) was solved as a problem of statics. Adding the consideration of dynamics of the strand brings one into the realm of so-called $G$-Strands \cite{HoIvPe2012}, for which the inexact trajectory planning problem may prove interesting. Finally, for applications in computational anatomy one must deal with infinite dimensional diffeomorphism groups. Extending the framework of this paper to infinite dimensions is therefore another challenge that lies ahead.

\paragraph{Acknowledgements}
We thank M. Bruveris, F. Gay-Balmaz, H. O. Jacobs, M. Leok, L. Noakes, T. S. Ratiu, J. Vankerschaver and F.-X. Vialard for several enlightening discussions of this material. DDH thanks the Royal Society for a Wolfson Research Merit Award and the European Research Council for Advanced Grant 267382.
\appendix
\section{Gradient calculation via adjoint equations}\label{App_AdEq}
In order to implement an efficient descent method for $\mathcal{J}_{d, \mbox{initial}}$ it is useful to have an expression for its gradient. One may obtain a gradient estimate using finite difference methods. One drawback lies with the inaccuracies inherent in the estimation. Moreover, if the dimension of the Lie algebra $\mathfrak{g}$ is large such estimations quickly become computationally costly. Both of these drawbacks can be circumvented in our case by the use of adjoint equations. In this way we obtain an exact expression for the gradient of  $\mathcal{J}_{d, \mbox{initial}}$, in a computationally efficient way. We will derive the system of adjoint equations now. For simplicity we treat only the special case where $\ell = \frac{1}{2} \left\|\xi^1\right\|^2_\gamma$, where $\gamma$ denotes an inner product on $\mathfrak{g}$ and $\|.\|$ the corresponding norm. Moreover we will only treat the left-action, right-reduction case, however the others can be obtained in the same way. We recall from \eqref{constraints}--\eqref{EP_node} the equations of motion  
\begin{align}
  &g_{k+1} = \tau(h\xi^0_k) g_k, \quad \xi^0_{k+1} = \xi^0_k + h(\mu^1_{k+1})^\sharp \label{eq_App_1}\\
  &\mu^1_{k+1} = \mu^1_k -h (D\tau_{-h\xi^0_{k}})^*\mu^0_{k+1}, \label{eq_App_2}\\
  &\mu^0_{k+1} = (D\tau^{-1}_{-h\xi^0_k})^* (D\tau_{h\xi^0_k})^* \left(\mu^0_k + \Delta_k(g_kQ_0)\right),\label{eq_App_3}
\end{align}
where we introduce functions $\Delta_k: Q \to \mathfrak{g}^*$ for $k = 0, \ldots, N$ defined as
\begin{align}
  \Delta_{N_i}(q) :=  \frac{d_{N_i}}{\sigma^2} J^Q({\rm{d}}_1 d(g_{N_i}Q_0, Q_{t_l}) \notag
\end{align}
when $k \in \{N_1, \ldots, N_l\}$ and $\Delta_k = 0$ otherwise. 

 Let us define an augmented functional $\mathcal{G}$, in which these equations are paired with Lagrange multipliers. These Lagrange multipliers will be denoted $(P_k^0, P^1_k, V^0_k, V^1_k) \in 2 \mathfrak{g}^* \times 2\mathfrak{g}$ for $k = 1, \ldots, N$. Let us introduce the shorthand notation $x$ representing the discrete path $(g_k, \xi^0_k, \mu_k^0, \mu^1_k)_{k=0}^N$ and $\lambda$ representing the ensemble of Lagrange multiplier $(P_k^0, P^1_k, V^0_k, V^1_k)_{k=1}^N$. The augmented functional $\mathcal{G}$ is given by
 \begin{align}
   \mathcal{G}(x, \lambda) &= h \Bigg[ \sum_{k=0}^{N-1} \frac{1}{2}\left\|(\mu_{k+1}^1)^\sharp\right\|_\gamma^2 + \left<P^0_{k+1}, \tau^{-1}(g_{k+1}g_k^{-1}) - h \xi^0_k\right> \notag \\
&\quad + \left<P^1_{k+1}, \xi^0_{k+1} - \xi^0_k - h (\mu^1_{k+1})^\sharp\right>  + \left<\mu^1_{k+1}-\mu^1_k + h(D\tau_{-h\xi^0_k})^*\mu^0_{k+1}, V^0_{k+1}\right> \notag \\
&\quad + \left<(D\tau_{-h\xi^0_k})^*\mu^0_{k+1} - (D\tau_{h\xi^0_k})^*(\mu^0_k + \Delta_k(g_kQ_0)), V^1_{k+1}\right> \Bigg] + \frac{1}{2 \sigma^2}\sum_{i=1}^l d^2(g_{N_i}Q_0, Q_{t_i}).\notag
 \end{align}
No constraints are assumed here, apart from the prescribed initial velocity $\xi^0_0$ and $g_0 = e$. It is clear that for any choice of Lagrange multipliers $\lambda$ we have $\mathcal{G}(x, \lambda) = \mathcal{J}_{d, \mbox{initial}}(\mu_0^0, \mu_0^1)$, as long as $x$ satisfies \eqref{eq_App_1}--\eqref{eq_App_3} for given initial values $\mu_0^0, \, \mu_0^1$. A tedious, but straightforward calculation shows that
\begin{align}
  \delta \mathcal{G}(x, \lambda) = -h \left<\delta \mu_0^1, V^0_1\right> - h \left<\delta \mu_0^0, D\tau_{h\xi_0^0}V_1^1\right>, \label{Vars}
\end{align}
if $x$ satisfies  \eqref{eq_App_1}--\eqref{eq_App_3} \emph{and} $\lambda$ is a solution of the \emph{adjoint equations}. We describe these now. We introduce functions $K^{\pm}: 2\mathfrak{g} \times \mathfrak{g}^* \to \mathfrak{g}^*$ by the defining relation
\begin{align}
  \left<K^{\pm}_{\xi, \mu}V, \rho\right> = \left.\frac{{\rm d}}{{\rm d}\varepsilon}\right|_{\varepsilon = 0} \left<(D\tau_{\pm h(\xi + \varepsilon \rho)})^*\mu, V\right>, \quad \mbox{for all } \xi,\, V,\, \rho \in \mathfrak{g},\, \mu \in \mathfrak{g}^*. \notag
\end{align}
Moreover, for $k= 0, \ldots, N$ we define functions $\mathcal{A}_k: Q \times \mathfrak{g} \to \mathfrak{g}^*$ by
\begin{align}
  \left<\mathcal{A}_k(q, \rho), \eta\right> =  \left.\frac{{\rm d}}{{\rm d}\varepsilon}\right|_{\varepsilon = 0} \left<\Delta_k(\exp(\varepsilon \eta) q), \rho\right>, \notag
\end{align}
for all $q \in Q$ and $\rho,\, \eta \in \mathfrak{g}$. The adjoint equations consist of conditions at the final time point,
\begin{align}
  &P^0_N = -h^{-1} (D\tau_{-h\xi^0_{N-1}})^* \Delta_N(g_NQ_0), \quad P^1_N = 0, \label{Init_1} \\
& V_N^0 = - (\mu_N^1)^\sharp, \quad V^1_N = -h V^0_N, \label{Init_2}
\end{align}
and the following equations for $k = 1,\ldots, N-1$,
\begin{align}
  &P^0_k = (D\tau_{-h\xi^0_{k-1}})^*\left[(D\tau^{-1}_{h\xi^0_k})^*P^0_{k+1} + \mathcal{A}_k(g_kQ_0, D\tau_{h\xi^0_k}V^1_{k+1}) - h^{-1}\Delta_k(g_kQ_0)\right]\label{Back_1}\\
&P^1_k = P^1_{k+1} + hP^0_{k+1} - hK^-_{\xi^0_k, \mu^0_{k+1}}V^0_{k+1} - K^-_{\xi^0_k, \mu^0_{k+1}}V^1_{k+1} + K^+_{\xi^0_k, \mu^0_k + \Delta_k(g_kQ_0)}V^1_{k+1} \label{Back_2}\\
  &V^0_k = V^0_{k+1} - (\mu^1_k)^\sharp + h(P^1_k)^\sharp, \label{Back_3}\\
  &V^1_k = -hV_k^0 + D\tau^{-1}_{-h\xi^0_{k-1}}D\tau_{h\xi^0_k}V^1_{k+1} \label{Back_4}.
\end{align}
These equations are posed backwards. That is, solving the adjoint equations entails initialising the Lagrange multipliers at time point $N$ according to \eqref{Init_1}--\eqref{Init_2} and then iterating backwards from $k = N$ to $k = 1$ using \eqref{Back_1}--\eqref{Back_4}.

We now obtain an expression for the gradient of $\mathcal{J}_{d, \mbox{initial}}$ from \eqref{Vars}. Indeed, let $(\mu_0^0(\varepsilon), \mu_0^1(\varepsilon))$ be a variation of initial conditions $(\mu_0^0, \mu_0^1)$, and let $x(\varepsilon)$ be the corresponding set of solutions to \eqref{eq_App_1}--\eqref{eq_App_3}. Let $\lambda$ be a solution to the adjoint equations \eqref{Back_1}--\eqref{Back_4} for $x = x(0)$, then
\begin{align}
  \delta \mathcal{J}_{d, \mbox{initial}} &=  \left.\frac{{\rm d}}{{\rm d}\varepsilon}\right|_{\varepsilon = 0} \mathcal{J}_{d, \mbox{initial}}(\mu_0^0(\varepsilon), \mu_0^1(\varepsilon)) =  \left.\frac{{\rm d}}{{\rm d}\varepsilon}\right|_{\varepsilon = 0} \mathcal{G}(x(\varepsilon), \lambda) \notag \\
&=  -h \left<\delta \mu_0^1, V^0_1\right>   - h \left<\delta \mu_0^0, D\tau_{h\xi_0^0}V_1^1\right> . \notag
\end{align}
From this we can read off the gradient,
\begin{align}
  \frac{\delta \mathcal{J}_{d, \mbox{initial}}}{\delta \mu_0^0} =  - h D\tau_{h\xi_0^0}V_1^1, \qquad   \frac{\delta \mathcal{J}_{d, \mbox{initial}}}{\delta \mu_0^1} = -h  V^0_1. \notag
\end{align}

\newpage
{\bibliographystyle{alpha}
\newcommand{\etalchar}[1]{$^{#1}$}
}

\end{document}